\documentclass{gOMS2e}

\usepackage{amssymb, amsmath, amsthm, latexsym}
\usepackage{fullpage, color}
\usepackage{hyperref}
\usepackage{natbib}
\usepackage{graphicx}
\usepackage{array}
\usepackage{multirow}
\usepackage{tabularx}
\usepackage{wrapfig}
\usepackage{comment}
\usepackage{tablefootnote}
\usepackage[dvipsnames]{xcolor}
\usepackage{algorithm,algorithmic}

\newtheorem{theorem}{Theorem}
\newtheorem{lemma}{Lemma}
\newtheorem{corollary}{Corollary}

\newtheorem{condition}{Condition}
\newtheorem{assumption}{Assumption}

\theoremstyle{definition}

\theoremstyle{remark}
\newtheorem{remark}{Remark}

\newcommand\numberthis{\addtocounter{equation}{1}\tag{\theequation}}

\usepackage{kamzolov}

\def\bk{\bar \kappa}

\def\dk#1{{\color{black}#1}} 
\def\aa#1{{\color{black}#1}} 
 
\begin{document}
\title{ Inexact Tensor Methods  and Their Application to Stochastic Convex Optimization
}

\author{
\name{Artem Agafonov\textsuperscript{1, 2}, Dmitry Kamzolov\textsuperscript{2, 1}, Pavel Dvurechensky\textsuperscript{3}, Alexander Gasnikov\textsuperscript{1, 4, 5}, and Martin Tak\'a\v{c}\textsuperscript{2}}
\affil{\textsuperscript{1}Moscow Institute of Physics and Technology, Dolgoprudny, Russia;
\textsuperscript{2}Mohamed bin Zayed University of Artificial Intelligence, Abu Dhabi, UAE;
\textsuperscript{3}Weierstrass Institute for Applied Analysis and Stochastics, Berlin, Germany;
\textsuperscript{4}Institute for Information Transmission Problems, Moscow, Russia;
\textsuperscript{5}HSE University, Moscow, Russia}
}



\maketitle

\begin{abstract}
    We propose general non-accelerated\footnote{The results for non-accelerated methods first appeared in December 2020 in the preprint \cite{agafonov2020inexact}.} and accelerated tensor methods under inexact information on the derivatives of the objective, analyze their convergence rate. Further, we provide conditions for the inexactness in each derivative that is sufficient for each algorithm to achieve a desired accuracy. As a corollary, we propose  stochastic tensor methods for convex optimization and obtain sufficient mini-batch sizes for each derivative.  
\end{abstract}


\section{Introduction}
The idea of using the derivatives of the order $p$ higher than two in optimization methods is known at least since 1970's \cite{hoffmann1978higher-order}. In numerical analysis it was used much earlier, see the student work of P.L.~Chebyshev \cite{Tchebyshev1951} and more recent reviews \cite{evtushenko2013,Trefethen2019approximation}. Despite theoretical advantages, the practical application of such tensor methods was limited until the recent work \cite{nesterov2019implementable} since each iteration of such methods includes minimization of a polynomial with degree larger than 3, which may be non-convex even for convex optimization problems. As it was shown in \cite{nesterov2019implementable}, if the minimization problem is convex, then in each iteration of the tensor method it is sufficient to minimize a convex polynomial, and, for $p=3$, this can be done with approximately the same cost as the step of the cubic-regularized Newton's method in the convex case \cite{nesterov2006cubic}. This leads to a method with faster convergence rate than that of the accelerated cubic-regularized Newton's method, and with approximately the same cost of each iteration. These ideas were further developed to obtain accelerated versions of tensor methods \cite{gasnikov2019near,nesterov2020inexact5} and very fast second-order methods \cite{nesterov2020superfast,kamzolov2020near,nesterov2020inexact5}. 

At the same time, many optimization problems in machine learning or image analysis have the form of minimization of a finite sum of functions and are solved by second-order methods \cite{zhang2015DiSCO,rodomanov2016superlinearly}. A more general setting,for which the finite-sum setting is a special case, is the setting of general stochastic optimization, for which second-order methods are also developed in the literature \cite{tripuraneni2018stochastic} for non-convex problems. Tensor methods for such problems are also developed, but with the focus on non-convex problems \cite{bellavia2018adaptive,lucchi2019stochastic}.  Thus, motivated by the lack of results on tensor methods for stochastic convex optimization, in this paper, we study stochastic convex optimization problems and develop tensor methods for this setting. We do this by a more general framework of Inexact Tensor Methods, which use inexact values of higher-order derivatives. First, we analyze such methods under a particular condition on the inexactness in the derivatives, and, then, we show how to satisfy this condition in the setting of stochastic optimization and propose a stochastic tensor methods for convex optimization. 

\subsection{Problem Statement}
We consider convex optimization problems of the following form
\begin{equation}
    \label{eq:PrSt}
    \min _{x \in \mathbb R^n} f(x)
\end{equation}
under inexact information on the derivatives of the objective $f$ up to order $p\geq 2$. We are motivated by two particular cases of this problem in the stochastic optimization setting, which we refer to as online and offline settings.
In the \textit{online setting} we assume that $f(x)$ is given as
\begin{equation}\label{eq:online_intro}
     f(x)=\mathbb{E}_{\xi \sim \mathcal{D}}[f(x ; \xi)],
\end{equation} 
where the random variable $\xi$ is sampled from a distribution $\mathcal{D}$, and an optimization procedure has access only to stochastic realizations of the function $f( x ; \xi)$ via its higher-order derivatives up to the order $p \ge 2$.
In the \textit{offline} setting the function $f$ has \textit{finite-sum} form: 
\begin{equation}\label{eq:offline_intro}
     f(x) = \frac1m \sum \limits_{i = 1}^m f_i(x)
\end{equation}
with available derivatives of the order $p \ge 2$ for each function $f_i$. Clearly, the offline setting can be considered as a particular case of the online setting if we set $\xi$ to be uniformly distributed over $i=1,...,m$. At the same time, we distinguish these two settings since the analysis of the proposed methods in the second setting can be made under weaker assumptions.

\subsection{Related Work}

\subsubsection{Second-order methods}
Beyond first-order methods, the most developed methods are maybe second-order methods, among which the closest to our setting are cubic-regularized Newton's methods originating from \cite{nesterov2006cubic}. This line of works includes the development of accelerated methods \cite{nesterov2008,monteiro2013accelerated,nesterov2018lectures}, extensions of trust region methods \cite{conn2000trust,cartis2011adaptive,cartis2011adaptive2,cartis2017improved,cartis2019universal},
methods with inexact evaluations of gradients and/or Hessians \cite{ghadimi2017secondorder,wang2018note,bellavia2020adaptive,doikov2020convex} with application to stochastic optimization in the online and offline settings. We especially point to \cite{ghadimi2017secondorder} which largely inspired our work.
Stochastic second-order methods for convex optimization  \cite{rodomanov2016superlinearly,kovalev2019stochastic,hanzely2020stochastic} and non-convex optimization \cite{tripuraneni2018stochastic,zhou2018stochastic,wang2020,park2020combining} have been extensively studied in the recent literature.
 The difference with our setting is that these works consider a particular case of $p=2$.

\subsubsection{Tensor methods}
To distinguish the general methods that use the derivatives up to the order $p > 2$, we refer to them as tensor methods. The idea of such methods was proposed quite long ago \cite{hoffmann1978higher-order}, and accelerated tensor methods for convex optimization were proposed in \cite{baes2009estimate}. Recent interest to this type of methods in convex optimization was probably motivated by the lower bounds obtained in \cite{arjevani2019oracle,agarwal2017lower}.
In \cite{nesterov2019implementable}, it was shown that appropriately regularized Taylor expansion of a convex function is also convex, which leads to implementability of such methods that minimize this regularized Taylor expansion. Accelerated tensor methods were also proposed in \cite{nesterov2019implementable}, yet with a remaining gap between the upper and lower complexity bounds. In the same paper, the author also shows how tensor methods with $p=3$ can be implemented by solving the auxiliary problem with the Bregman projected gradient method in the relative smoothness setting.
Near-optimal, i.e. with optimal up to a logarithmic factor iteration complexity, tensor methods for convex optimization were recently proposed in a number of works
\cite{gasnikov2018global,gasnikov2019reachability,gasnikov2019near,gasnikov2019optimal,bubeck2019near, jiang2019optimal,gasnikov2020book,nesterov2020inexact5,dvinskikh2020accelerated}.
These developments allowed to propose faster second-order methods via implementing third-order methods with inexact third derivative \cite{nesterov2020superfast,kamzolov2020near,nesterov2020inexact5}, which lead to an improvement of the complexity bound from $O({\e^{-1/3.5}})$ to $O({\e^{-1/5}})$.
Stochastic tensor methods were developed for non-convex smooth enough problems in \cite{bellavia2018adaptive,lucchi2019stochastic}. 
We are not aware of an analysis of tensor methods for stochastic convex smooth enough problems.

Our analysis is based on inexact versions of tensor methods for convex optimization, which use inexact derivatives of the order up to $p \geq 2$. First-order methods with inexact gradients are well-developed in the literature, see, for example, \cite{polyak1987introduction,d2008smooth,devolder2013exactness,dvurechensky2016stochastic,gasnikov2016stochasticInter,cohen2018acceleration,dvinskikh2019decentralized,stonyakin2019inexact} and references therein. Some results on inexact second-order methods are listed above. The paper \cite{nesterov2020inexact5} proposes an analysis of third-order method for convex optimization with inexact third derivative. In \cite{bellavia2018adaptive} the authors analyze inexact tensor methods for non-convex optimization.
The general theory of inexact tensor methods of an arbitrary order $p$ for convex problems still has to be developed and we make a step in this direction.

\subsection{Our contribution}
Motivated by stochastic optimization methods, we propose and analyze inexact tensor methods of the general order $p \geq 2$ for convex optimization problems. 
The idea of these algorithms is to use inexact values of the derivatives up to the order $p$ and to construct a regularized inexact Taylor expansion. More precisely, we assume that inexact derivative $G_{x, i}$ satisfies the following condition (see Condition \ref{cnd:sampling} below):
$$
\left\|\left(G_{x, i} - \nabla^i f(x)\right) [y - x]^{i-1}\right\| \leq \delta_i \|y-x\|^{i-1}, 
$$
for all $y$ and some $\delta_i \geq 0$, and $i=1, 2, \ldots, p$. Based on the $p$-th order inexact Taylor polynomial we build Inexact Tensor Model and use it to construct Inexact Tensor Method and its accelerated version. For the Inexact Tensor Method, we obtain the following convergence rate result:
$$
f(x_T) - f(x^{\ast}) \leq O \left(\sum \limits_{i=1}^p \frac{\delta_i}{T^{i-1}}\max_{x \in \mathcal{L}^\prime(x_0)}\|x-x^{\ast}\|^i + \frac{L_p }{T^p}\max_{x \in \mathcal{L}^\prime(x_0)}\|x-x^{\ast}\|^{p+1} \right),
$$ 
where $\mathcal{L}^\prime(x_0)$ is defined in \eqref{eq:lebesgue_prime}, $L_p$ is the Lipschitz constant of the $p$-th derivative, and $T$ is the iteration counter. 

For the Accelerated Inexact  Tensor Method we obtain the following convergence rate result:
$$
f(x_T) - f(x^{\ast}) \leq O \left( \delta_1 \bar{R}+ \sum \limits_{i=2}^p \frac{\delta_i}{T^{i}}\|x^*-x_0\|^i + \frac{L_p }{T^{p+1}}\|x^*-x_0\|^{p+1} \right),
$$
where $\bar{R}$ is defined in \eqref{eq:assumption_bound_x}.
Based on the above convergence rate estimates, we obtain conditions on the inexactnesses $\delta_i$, $i=1, 2, \ldots, p$ of the derivatives that are sufficient for the proposed inexact methods to have the same convergence rate as their exact counterparts.

As a corollary, we propose stochastic non-accelerated and accelerated tensor methods for stochastic convex optimization problems in the online and offline settings. The idea of the algorithms is to sample in each iteration mini-batches of derivatives up to the order $p$ and use them to construct a regularized inexact Taylor expansion. Using Hoeffding Tensor concentration inequality, we estimate mini-batch sizes for each derivative that are sufficient for the whole method to achieve accuracy $\e$ in the same number of iterations as the exact method requires.
For the non-accelerated method these mini-batch sizes are $O\left(\e^{-2(p+1-i) /p}\right)$, $i=1, 2, \ldots, p$ and for the  accelerated method they are $O(\e^{-2})$ for $i=1$ and $O\left(\e^{-2(p+1-i)/(p+1)}\right)$, $i=2, \ldots, p$. 
Interestingly, the higher is the derivative order, the smaller batch size turns out to be sufficient.

We also consider a particular case $p=3$, for which, in the spirit of \cite{nesterov2019implementable}, we describe how to implement the resulting inexact tensor methods.

\subsection{Paper Organization}
The remaining part of the paper is organized as follows. In Section \ref{sec:prel}, we introduce main notations and assumptions. Then, in Section \ref{sec:model} we present an inexact tensor model of the objective function, prove its convexity, and show that it majorizes the objective function. 
Section \ref{sec:method} is dedicated to the Inexact Tensor Method itself and its convergence rate analysis. 

Next, we introduce and analyze the Accelerated Inexact Tensor Method in Section \ref{sec:acceleration}. In Section $\ref{sec:implement}$, we discuss some implementation details. Finally, in Section $\ref{sec:smpl}$ we apply the general inexact methods to introduce Stochastic Tensor Methods. 

\section{Preliminaries}\label{sec:prel}
We denote the $i$-th directional derivative of function $f$ at $x$ along directions $s_1, \ldots, s_i \in \mathbb R^n$ as  
$$\nabla^i f(x)[s_1, \ldots, s_i].$$
For example, 
$\nabla f(x)[s_1] = \langle \nabla f(x), s_1\rangle$ and $\nabla^2 f(x)[s_1, s_2] = \langle \nabla^2 f(x)s_1, s_2\rangle$. If all directions are the same we write $\nabla^i f(x)[s]^i.$
For a $p$-th order tensor $T$, we denote by $\|T\|$ its tensor norm recursively induced \cite{cartis2017improved} by the Euclidean norm on the space of $p$-th order tensors: 
$$\|T\| = \max \limits_{\|s_1\| =  \ldots = \| s_p\| = 1} \{|T[s_1, \ldots, s_p]| \},$$
where we denote by $\|\cdot\|$ the standard Euclidean norm.
Throughout the paper we make the following assumption.
\begin{assumption}\label{as:lip}
    Function $f$ is convex,  $p$ times differentiable on $\mathbb{R}^n$, and its $p$-th derivative is  Lipschitz continuous, i.e. for all $
    x, y \in \mathbb{R}^n$
    $$\|\nabla^p f(x) - \nabla^p f(y)\| \leq L_p \|x - y\|.$$ 
\end{assumption}

Following the previous works, we construct tensor methods based on the $p$-th order Taylor approximation of the function $f(x)$, which can be written as follows: 
\begin{equation}\label{eq:taylor}
    \Phi_{x,p}(y) \stackrel{\text { def }}{=} f(x)+\sum_{i=1}^{p} \frac{1}{i !} \nabla^{i} f(x)[y-x]^{i}, \quad y \in \mathbb{R}^n.
\end{equation}
The full Taylor expansion of $f$ requires computing all the derivatives up to the order $p$, which can be expensive to calculate. Thus, it is natural to use some approximations $G_{x, i}$ for the derivatives $\nabla^i f\left(x\right)$, $i=1,\ldots,p$ and to construct an inexact $p$-th order Taylor expansion of the objective: 
\begin{equation}\label{eq:approx_taylor}
    \phi_{x,p}(y)=f\left(x\right)+\sum \limits_{i = 1}^p  \frac{1}{i!}G_{x, i}[y-x]^i, 
\end{equation}
where $G_{x, i}$ satisfies the following
\begin{condition}\label{cnd:sampling}
    Given the inexactness levels $\delta_i \geq 0$ for $i=1, \ldots, p$, for all $x \in \mathbb{R}^n$ the approximate derivatives $G_{x,i}$ satisfy for all $y \in \mathbb{R}^n$ the following inequalities:
    \begin{equation}\label{eq:sampling_cnd}
        \|(G_{x,i} - \nabla^i f({x}))[y - x]^{i-1}\| \le \delta_i \|y - x\|^{i-1}, i = 1, \ldots, p. \\
    \end{equation}
\end{condition}

Below, we first analyze inexact tensor methods under different conditions on the accuracy of the approximations for the derivatives. Then, motivated by stochastic optimization problems \eqref{eq:PrSt}, \eqref{eq:online_intro}, we focus on stochastic approximations of the derivatives
through sampling. More precisely, for $\mathcal{S}_1, \mathcal{S}_2, \ldots, \mathcal{S}_p$ being sample sets, we construct the sampled approximations as
\begin{equation}\label{eq:sampled_derivs}
\begin{aligned}
    {G}_{{x}, i} = \frac{1}{\left|\mathcal{S}_i\right|} \sum_{j \in \mathcal{S}_i} \nabla^i f\left({x}, \xi_j\right), \quad i=1,..,p.
    \end{aligned}
\end{equation}
In Section \ref{sec:smpl} we show how to choose the size of the sample sets $\mathcal{S}_i$ so that our algorithms achieve a desired accuracy with high probability. 

In our methods we use the following $p$-th-order prox functions: 
\begin{equation*}
    d_p(x) = \frac1p\|x\|^p, ~ p \geq 2.
\end{equation*}
Note that
\begin{equation*}
        \nabla d_{p}(x) =\|x\|^{p-2}  x, 
\end{equation*}
\begin{equation}\label{eq:ppf2deriv}
    \nabla^{2} d_{p}(x) =(p-2)\|x\|^{p-4}  x x^{*} +\|x\|^{p-2}  I  \succeq\|x\|^{p-2}  I , 
\end{equation}
where $I$ is the identity matrix in $\mathbb{R}^n$.

As it is shown, e.g. in \cite{nesterov2019implementable}, Assumption \ref{as:lip} allows to control the quality of the approximation of the objective $f$ by its Taylor polynomial: 
\begin{equation}\label{eq:taylor_bnd}
    |f({y}) - \Phi_{{x},p}(y)| \leq \frac{L_p}{(p+1)!}\|y - x\|^{p+1}, \;x, y \in  \mathbb{R}^n.
\end{equation}
If $p \geq 2$, the same can be done with the first and second derivatives:
\begin{equation}\label{eq:taylor_grad_bnd}
    \|\nabla f({y}) - \nabla \Phi_{{x},p}(y)\| \leq \frac{L_p}{p!}\|y - x\|^{p}, \; x, y \in  \mathbb{R}^n,
\end{equation}
\begin{equation}\label{eq:taylor_hess_bnd}
    \|\nabla^2 f({y}) - \nabla^2 \Phi_{x,p}(y)\| \leq \frac{L_p}{(p-1)!}\|y - x\|^{p-1}, \; x, y \in  \mathbb{R}^n.
\end{equation}

\section{Inexact Tensor Model}\label{sec:model}
In this section, we analyze inexact Taylor approximation \eqref{eq:approx_taylor} under Condition \ref{cnd:sampling}.
First, we obtain under  Condition \ref{cnd:sampling} approximation bounds in the spirit of \eqref{eq:taylor_bnd}, \eqref{eq:taylor_grad_bnd}, \eqref{eq:taylor_hess_bnd}. Then, based on these bounds, we construct a regularized inexact Taylor polynomial and show that it is a global upper bound for the objective function $f$. The latter is the key to construct the proposed Inexact Tensor Method, which we analyze in the next section. 

The following lemma gives a counterpart of \eqref{eq:taylor_bnd} when the inexact derivatives are used and shows that we can bound the residual between function $f$ and  the $p$-th order inexact Taylor polynomial $\phi_{x,p}(y)$ defined in \eqref{eq:approx_taylor}.
\begin{lemma}
    \label{lm:fun_bnd1}
     Assume that Condition \ref{cnd:sampling} is satisfied. Then, for any $x,y \in \mathbb{R}^n$, we have
    \begin{equation}
    \begin{split}\label{eq:func_bnd}
         |f( y) - \phi_{{x},p}(y)|  
         \leq \sum \limits_{i = 1}^p \frac{ \delta_i}{i!}\|y - x\|^i + \frac{L_p}{(p+1)!}\|y - x\|^{p+1}.
     \end{split}
\end{equation}
\end{lemma}
\begin{proof}
   For any $x, y \in \mathbb{R}^n$:
   $$|f( y) - \phi_{x,p}(y)| \leq |f(y) - \Phi_{x,p}(y)| + |\Phi_{x,p}(y) - \phi_{x,p}(y)| \stackrel{\eqref{eq:taylor_bnd}}{\leq} \frac{L_p}{(p+1)!}\|y - x\|^{p+1} + |\Phi_{x,p}(y) - \phi_{x,p}(y)|.$$
   Let us bound the second term in the right hand side of the inequality above:
   \begin{gather*}
       |\Phi_{x,p}(y) - \phi_{x,p}(y)|  \stackrel{\eqref{eq:taylor}, \eqref{eq:approx_taylor}}{=} \left\vert \sum \limits_{i = 1}^p \frac{1}{i!}({G}_{x, i} - \nabla^i f(x))[y - x]^i \right\vert \\
       \leq \sum \limits_{i = 1}^p \frac{1}{i!}\|({G}_{x, i} - \nabla^i f(x))[y - x]^{i-1}\|\|y - x\| \stackrel{\eqref{eq:sampling_cnd}}{\leq} \sum \limits_{i = 1}^p \frac{ \delta_i}{i!} \|y - x\|^i. 
   \end{gather*}
   Combining  both inequalities above, we finis the proof.
\end{proof}

The following lemma gives a counterpart of \eqref{eq:taylor_grad_bnd}, \eqref{eq:taylor_hess_bnd} for the inexact Taylor polynomial $\phi_{x,p}(y)$ defined in \eqref{eq:approx_taylor}.
\begin{lemma}
 \label{lm:grad_bnd1}
 Assume that Condition \ref{cnd:sampling} is satisfied. Then, for any $x,y \in  \mathbb{R}^n $, we have
\begin{equation}
    \begin{split}\label{eq:1deriv_bnd}
         \|\nabla f( y) - \nabla \phi_{x,p}(y)\| 
         \leq \sum \limits_{i = 1}^{p} \frac{ \delta_i}{(i-1)!}\|y - x\|^{i-1} + \frac{L_p}{p!}\|y - x\|^{p},
    \end{split}
\end{equation}
\begin{equation}
    \begin{split}\label{eq:2deriv_bnd}
         \|\nabla^2 f( y) - \nabla^2 \phi_{x,p}(y)\| \leq \sum \limits_{i = 2}^{p} \frac{\delta_i}{(i - 2)!}\|y - x\|^{i-2} + \frac{L_p}{(p-1)!}\|y - x\|^{p-1},
    \end{split}
\end{equation}
where we use the standard convention $0!=1$.
\end{lemma}
\begin{proof}
    First, we prove the bound for the first derivatives. For any $x, y \in \mathbb{R}^n$:
    \begin{gather*}
        \|\nabla f( y) - \nabla \phi_{x,p}(y)\| \leq \|\nabla f(  y ) - \nabla \Phi_{x,p}(y)\| + \|\nabla \Phi_{x,p}(y) - \nabla \phi_{x,p}(y)\| \\
        \stackrel{\eqref{eq:taylor_grad_bnd}}{\leq}  \frac{L_p}{p!}\|y - x\|^{p} + \|\nabla \Phi_{x,p}(y) - \nabla \phi_{x,p}(y)\|.
    \end{gather*}
    Let us bound the second term in the right hand side of the inequality above: 
    \begin{gather*}
       \|\nabla \Phi_{x,p}(y) - \nabla \phi_{x,p}(y)\|  \stackrel{\eqref{eq:taylor}, \eqref{eq:approx_taylor}}{=} \left\| \sum \limits_{i = 1}^p \frac{1}{(i - 1)!}({G}_{x, i} - \nabla^i f(x))[y - x]^{i-1} \right\| \\
       \leq \sum \limits_{i = 1}^p \frac{1}{(i-1)!}\|({G}_{x, i} - \nabla^i f(x))[y - x]^{i-1}\| \stackrel{\eqref{eq:sampling_cnd}}{\leq} \sum \limits_{i = 1}^p \frac{ \delta_i}{(i - 1)!}\|y - x\|^{i-1}. 
   \end{gather*}
   Combining  both inequalities above, we get $\eqref{eq:1deriv_bnd}$.
    
    Now, we obtain the bound for the second derivatives.  For any $x, y \in \mathbb{R}^n$, we have
    \begin{gather*}
        \|\nabla^2 f( y) - \nabla^2 \phi_{x,p}(y)\| \leq \|\nabla^2 f(y) - \nabla^2 \Phi_{x,p}(y)\| + \|\nabla^2 \Phi_{x,p}(y) - \nabla^2 \phi_{x,p}(y)\| \\
        \stackrel{\eqref{eq:taylor_hess_bnd}}{\leq}  \frac{L_p}{(p-1)!}\|y - x\|^{p-1} + \|\nabla^2 \Phi_{x,p}(y) - \nabla^2 \phi_{x,p}(y)\|.
    \end{gather*}
    Let us bound the second term in the right hand side of the inequality above:  
    \begin{gather*}
       \|\nabla^2 \Phi_{x,p}(y) - \nabla^2 \phi_{x,p}(y)\|  \stackrel{\eqref{eq:taylor}, \eqref{eq:approx_taylor}}{=} \left\| \sum \limits_{i = 2}^p \frac{1}{(i - 2)!}({G}_{x, i} - \nabla^i f(x))[y - x]^{i-2} \right\| \\
       \leq \sum \limits_{i = 2}^p \frac{1}{(i-2)!}\|({G}_{x, i} - \nabla^i f(x))[y - x]^{i-2}\| \leq \sum \limits_{i = 2}^p \frac{ \delta_i}{(i - 2)!}\|y - x\|^{i-2}, 
   \end{gather*}
   where the last inequality is valid due to conditions $\eqref{eq:sampling_cnd}$ and the definition of the tensor 
   norm. Indeed, for $i=2,\ldots,p$, we have  
   \begin{gather*}
     \|({G}_{x, i} - \nabla^i f(x))[y - x]^{i-2}\| = \max_{\|s_1\|=\|s_2\|=1} |({G}_{x, i} - \nabla^i f(x))[y - x]^{i-2}[s_1,s_2]| \\
     = \max_{y \in \mathbb{R}^n, y \ne x} \frac{|({G}_{x, i} - \nabla^i f(x))[y - x]^{i-2}[y - x,y - x]|}{\|y-x\|^2} \leq
   \max_{y \in \mathbb{R}^n, y \ne x} \frac{\|({G}_{x, i} - \nabla^i f(x))[y - x]^{i-1}\|}{\|y-x\|} \\
   \stackrel{\eqref{eq:sampling_cnd}}{\leq}  \delta_i\|y - x\|^{i - 2} 
   \end{gather*}
   Combining the above bounds, we obtain $\eqref{eq:2deriv_bnd}$.
   
\end{proof}

We finish this section by constructing a global upper bound for the objective $f$ based on the inexact Taylor expansion $\phi_{{x, p}}$ which is regularized by functions $d_i$, $i=2,...,p$.
\begin{theorem}\label{thm:model_cnvxty}
    Assume that Condition \ref{cnd:sampling} is satisfied. Then, for any $\gamma \in \mathbb{R}, ~ x, y \in \mathbb{R}^n$, and $\sigma \geq L_p$  the function
    \begin{equation}\label{eq:model}
    \omega_{{x, p}}(y) = \phi_{x,p}(y)  
      + \frac{\delta_1}{2\gamma} + \gamma \delta_1 d_2(y - x)  + \sum \limits_{i =  2}^p \frac{ \delta_i}{(i -  2)!} d_i(y - x)  + \frac{\sigma}{(p-1)!}d_{p+1}(y - x)
     \end{equation}  
     is convex and  majorizes the function $f$: 
     \begin{equation}\label{eq:model_major}
         f( y) \leq \omega_{{x, p}}(y), ~~ {x, y \in   \mathbb{R}^n}.
     \end{equation}
\end{theorem}
\begin{proof}

For any $x,y \in  \mathbb{R}^n$ and $h \in \R^n$:
    \begin{equation*}
        \begin{split}
             \left\langle\left(\nabla^{2} f( y)-\nabla^{2} \phi_{x,p}(y)\right) h, h\right\rangle & \leq \left\|\nabla^{2} f(y)-\nabla^{2} \phi_{x,p}(y)\right\| \cdot\|h\|^{2} \\ & \stackrel{\eqref{eq:2deriv_bnd}}{\leq}\left( \sum \limits_{i = 2}^{p} \frac{ \delta_i }{(i - 2)!}\|y - x\|^{i-2} + \frac{L_p}{(p-1)!}\|y - x\|^{p-1}\right)\|h\|^{2}. 
        \end{split}
    \end{equation*}
   Whence, using the convexity of $f$, we have
    \begin{equation}\label{eq:hessian_bound}
        0 \preccurlyeq \nabla^2 f(y) \preccurlyeq \nabla^{2} \phi_{x, p}(y) +
        \sum \limits_{i = 2}^p \frac{ \delta_i}{(i - 2)!}\|y - x\|^{i-2} I +  \frac{L_p}{(p-1)!}\|y - x\|^{p-1} I
    \end{equation}
     \begin{gather*}
        \overset{\eqref{eq:ppf2deriv}}{\preccurlyeq} \nabla^2 \phi_{x,p}(y) +\sum \limits_{i = 2}^p \frac{ \delta_i}{(i - 2)!} \nabla^{2} d_{i}(y-x) I +  \frac{\sigma}{(p-1)!}  \nabla^{2} d_{p+1}(y-x) I =
        \nabla^2  \tilde{\omega}_{{x, p}}(y),
    \end{gather*}
   where we used that $\sigma \geq L_p$ and defined 
   \begin{equation*}
       \tilde{\omega}_{{x, p}}(y)=\phi_{x,p}(y)  
     + \sum \limits_{i = 2}^p \frac{ \delta_i}{(i - 2)!}d_i(y - x)  + \frac{\sigma}{(p-1)!}d_{p+1}(y - x).
   \end{equation*}
Thus, we conclude that $\tilde{\omega}_{{x, p}}(y)$ is convex. Further, $\omega_{{x, p}}(y)=\tilde{\omega}_{{x, p}}(y) + \frac{\delta_1}{2\gamma} + \gamma \delta_1 d_2(y-x)$ is convex as a sum of three convex functions.

Finally, we prove inequality \eqref{eq:model_major}: 
    \begin{gather}
        f(y) \stackrel{\eqref{eq:func_bnd}}{\leq} \phi_{x,p}(y) +  \sum \limits_{i = 1}^p \frac{ \delta_i}{i!}\|y - x\|^i + \frac{L_p}{(p+1)!}\|y - x\|^{p+1} \notag  \\ 
        =  \phi_{x,p}(y)  + \delta_1\|y-x\| +  \sum \limits_{i = 2}^p \frac{\ \delta_i}{(i-1)!}d_i(y - x) + \frac{L_p}{p!}d_{p+1}(y - x) \notag \\
        \leq   \phi_{x,p}(y)  +  \delta_1\|y-x\| + \sum \limits_{i =  2}^p \frac{ \delta_i}{(i- 2)!}d_i(y - x) + \frac{\sigma}{(p-1)!}d_{p+1}(y - x) \notag \\ 
        \leq \phi_{x,p}(y)  + \frac{\delta_1}{2\gamma} + \gamma \delta_1 d_2(y-x) + \sum \limits_{i =  2}^p \frac{ \delta_i}{(i - 2)!}d_i(y - x) + \frac{\sigma}{(p-1)!}d_{p+1}(y - x)
        = \omega_{{x, p}}(y),\label{eq:model_constr}
    \end{gather}
    where the last inequality makes the model smooth by using $\|x\|\leq \frac{1}{2\gamma} + \frac{\gamma}{2}\|x\|^2$. We write Remark~\ref{remark1} for additional details.
\end{proof}

Thus, we have constructed a regularized inexact Taylor polynomial $\omega_{{x, p}}(y)$ defined in \eqref{eq:model} as a model of the objective $f(y)$. Theorem \ref{thm:model_cnvxty} claims that this model satisfies two main conditions: 
\begin{itemize}
    \item Model $\omega_{{x, p}}(y)$ is a global upper bound for the function $f$:
        \begin{equation}\label{eq:model_major_simple}
            f( y) \stackrel{\eqref{eq:model_major}}{\leq} \omega_{{x, p}}(y), \; x,y \in \mathbb{R}^n.
        \end{equation} 
    \item Model $\omega_{{x, p}}(y)$ is convex.
\end{itemize} 
In the next section we use the model $\omega_{{x, p}}(y)$ to construct Inexact Tensor Method.

\section{Inexact Tensor Method}\label{sec:method}
Based on the inexact regularized Taylor polynomial $\omega_{x, p}(y)$ defined in the previous section, in this section, we present the Inexact Tensor Method (ITM). Each step of this algorithm uses minimization of our model $\omega_{{x, p}}(y)$ to make a step from a point $x$. 
To be more precise, we define an operator $S(x)$ as
\begin{equation}\label{eq:argmin}
    S(x) = \argmin_{y \in \mathbb{R}^n} \omega_{{x, p}}(y).
\end{equation}
The resulting Inexact Tensor Method is listed as Algorithm \ref{alg:inexact}.

\begin{algorithm}
  \caption{Inexact Tensor Method (ITM)}\label{alg:inexact}
  \begin{algorithmic}[1]
    \STATE \textbf{Input:} convex function $f$ such that $\nabla^p f$ is  $L_p$-Lipschitz; $x_0$ is starting point; constant $\sigma \geq L_p$.
    \FOR{$t \geq 0$} 
    \STATE Call the inexact oracle to compute ${G}_{x_t, i}$ for $i = 1, \ldots, p$ such that Condition $\ref{cnd:sampling}$ is satisfied.
    \STATE Make the step:
        \begin{equation}\label{eq:tensor_prox_algo}
                x_{t + 1} = \argmin_{y \in \mathbb{R}^n} \omega_{{x_t, p}}(y).
        \end{equation}
    \ENDFOR
  \end{algorithmic}
\end{algorithm}
\aa{
In view of inequality \eqref{eq:model_major}, the process $S(x)$ guarantees that $f(S(x)) < f(x) + \frac{\delta_1}{2\gamma}$, i.e. Algorithm \ref{alg:inexact} is almost monotone, i.e. is monotone up to accuracy $\frac{\delta_1}{2\gamma}$. Indeed, 
$$f(x_{t+1}) \stackrel{\eqref{eq:model_major}}{\leq} {\omega}_{x_t, p}(x_{t+1}) \stackrel{\eqref{eq:tensor_prox_algo}}{\leq} {\omega}_{x_t, p}(x_{t}) = f(x_t) + \frac{\delta_1}{2\gamma}.$$
Therefore, for any iteration $T\geq 1$, we have $f(x_T) \leq f(x_0) + \frac{T\delta_1}{2\gamma}$. We define \begin{equation}\label{eq:lebesgue_prime}
    \mathcal{L}^\prime (x_0) = \lb x \vert f(x) \leq f(x_0) + \frac{T\delta_1}{2\gamma}\rb.
\end{equation}
}
\begin{remark}
\label{remark1}
    The approximate monotonicity is caused by our construction of the model $\omega_{{x, p}}(y)$ which uses in \eqref{eq:model_constr} the inequality $\|y - x\| \leq \frac{1}{2\gamma} + \frac{\gamma}{2}\|y - x\|^2$ which holds for any $\gamma >0$. This is needed to obtain a smooth model which is easier to minimize. Otherwise, we can construct a non-smooth model
    $$\bar{\omega}_{x, p}(y) = \phi_{x,p}(y) + \delta_1\|y-x\| + \sum\limits_{i=1}^p \frac{\delta_i}{(i-2)!}d_i(y-x)+\frac{\sigma}{(p-1)!}d_{p+1}(y-x)
    $$
    and minimize it in each iteration of the method. Such constructed modification of the ITM becomes monotone. Indeed,
    $$f(x_{t+1}) \leq \bar{\omega}_{x_t, p}(x_{t+1}) \leq \bar{\omega}_{x_t, p}(x_{t}) = f(x_t).$$
\end{remark}

The next technical result is used to prove the convergence rate theorem for ITM.
\begin{lemma} 
\label{Lm:ITM_technical}
Assume that Condition \ref{cnd:sampling} is satisfied. Then, for any $\gamma \in \mathbb{R}, ~ x \in  \mathbb{R}^n$, and $\sigma \geq L_p$, we have
    \begin{equation} \label{eq:min_lemma}
        f(S(x)) \leq \min \limits_{y \in \mathbb{R}^n} \left\{ f(y)    + \tfrac{\delta_1}{\gamma} + 2\gamma \delta_1 d_2(y-x) + \sum \limits_{i = 2}^p  \tfrac{i\delta_i}{(i-1)!} d_i(y - x) +  \tfrac{L_p + p\sigma }{ p!}d_{p+1}(y -x) \right\}.
    \end{equation}
\end{lemma}
\begin{proof}
\begin{gather*}
    f(S(x)) \stackrel{\eqref{eq:model_major_simple}}{\leq} \omega_{{x, p}}(S(x)) \\
    \stackrel{\eqref{eq:argmin},\eqref{eq:model}}{=}  \min \limits_{y \in \mathbb{R}^n} \left \{ \phi_{x,p}(y)  + \tfrac{\delta_1}{2\gamma} + \gamma \delta_1 d_2(y-x) + \sum \limits_{i = 2}^p \tfrac{ \delta_i}{(i- 2)!} d_i(y - x) + \tfrac{\sigma }{ (p-1)!}d_{p+1}(y - x) \right \} \\
    \stackrel{\eqref{eq:func_bnd}}{\leq}  \min \limits_{y \in \mathbb{R}^n} \left\{ f(y)   + \tfrac{\delta_1}{2\gamma} + \delta_1\|y-x\| + \gamma \delta_1 d_2(y-x) + \sum \limits_{i = 2}^p  \tfrac{i\delta_i}{(i-1)!}d_i(y - x) +  \tfrac{L_p + p\sigma }{p!}d_{p+1}(y -x) \right\} \\
    \stackrel{\|x\|\leq \frac{1}{2\gamma} + \frac{\gamma}{2}\|x\|^2}{\leq} \min \limits_{y \in \mathbb{R}^n} \left\{ f(y)   + \tfrac{\delta_1}{\gamma} + 2 \gamma \delta_1 d_2(y-x) + \sum \limits_{i = 2}^p  \tfrac{i\delta_i}{(i-1)!}d_i(y - x) +  \tfrac{L_p + p\sigma }{p!}d_{p+1}(y -x) \right\} .  
\end{gather*}
\end{proof}

We are now in a position to estimate the convergence rate of the Inexact Tensor Method.

\begin{theorem}\label{thm:convergence}
    Let $f$ be convex function and let Assumption \ref{as:lip} 
    and Condition \ref{cnd:sampling} hold. Let also $\sigma \geq L_p$. Then, for any $\gamma \in \mathbb{R}^n$, after $T+1$ iterations of Algorithm \ref{alg:inexact}, we have the following bound for the objective residual in problem \eqref{eq:PrSt}: 
    \begin{align}
            f(x_{{T}+1}) - f(x^{\ast})&\leq \frac{\delta_1}{\gamma}(T + 1) + \frac{\gamma\delta_1\aa{D}^2(p+1)^2}{T+p+1} \notag \\
            & + \sum \limits_{i = 2}^p \frac{\delta_i \aa{D}^i}{(i-1)!} \frac{(p+1)^{i}}{({T}+p+1)^{i-1}} 
             + \frac{(L_p + p\sigma )\aa{D}^{p+1}}{{(p+1)!}} \frac{(p+1)^{p + 1}}{({T}+p+1)^{p}}.  \label{eq:convergence}
    \end{align}
    By introducing $O(\cdot)$ notation and fixing $\gamma = O(T/\aa{D})$, we get
    \begin{equation}\label{eq:convergence_O}
            f(x_{{T}+1}) - f(x^{\ast})\leq
             \sum \limits_{i = 1}^p O\ls \frac{\delta_i \aa{D}^i}{T^{i-1}} \rs
             + O\ls\frac{L_p \aa{D}^{p+1}}{T^{p}}\rs,
    \end{equation}
    \aa{where}
    \begin{equation}
    \label{D_lebeg}
        \aa{D = \max\limits_{x\in \mathcal{L}^\prime(x_0)} \|x - x_0\|.}
    \end{equation}
   
\end{theorem}
\begin{proof}
    Applying Lemma \ref{Lm:ITM_technical} for any $t \geq 0$, we obtain
    \begin{equation*}
        \begin{split}
            f(x_{t+1}) \stackrel{\eqref{eq:min_lemma}}{\leq} 
            \min \limits_{y \in \mathbb{R}^n} \left\{ f(y)    + \frac{\delta_1}{\gamma} + 2\gamma \delta_1 d_2(y-x_t)   + \sum \limits_{i =  2}^p  \frac{i\delta_i }{(i-1)!}d_i(y - x_t) +  \frac{L_p + p\sigma }{ p !}d_{p+1}(y - x_t)   \right \}  \\
            \stackrel{\eqref{D_lebeg}}{\leq}  \min \limits_{\alpha_t \in [0, 1]} \left \{ f(x_{t} + \alpha_t (x^{\ast} - x_t)) + \frac{\delta_1}{\gamma} + \gamma \delta_1 (\alpha_t \aa{D})^2 + \sum \limits_{i = 2}^p \frac{\delta_i }{(i-1)!}(\alpha_t \aa{D})^i +  \frac{L_p + p\sigma }{{(p+1)}!}(\alpha_t \aa{D})^{p+1}   \right\} \\
            \leq \min \limits_{\alpha_t \in [0, 1]} \left \{ (1-\al_t)f(x_t) + \alpha_t f(x^{\ast}) + \frac{\delta_1}{\gamma} + \gamma \delta_1 (\alpha_t \aa{D})^2 + \sum \limits_{i = 2}^p\frac{\delta_i }{(i-1)!}(\alpha_t \aa{D})^i +  \frac{L_p + p\sigma }{{(p+1)}!}(\alpha_t \aa{D})^{p+1}  \right\}.
        \end{split}
    \end{equation*}
    Therefore, subtracting $f(x^{\ast})$ from both sides, we obtain, for any $\alpha_t \in [0,1]$
        \begin{equation}\label{eq:conv_to_sum}
            \begin{split}
                f(x_{t+1}) - f(x^{\ast}) &\leq (1-\alpha_t)(f(x_t) - f(x^{\ast})) \\
                &+ \frac{\delta_1}{\gamma} + \gamma \delta_1 (\alpha_t \aa{D})^2 + \sum \limits_{i = 2}^p \frac{\delta_i }{(i-1)!}(\alpha_t \aa{D})^i +  \frac{L_p + p\sigma }{{(p+1)!}}(\alpha_t \aa{D})^{p+1}.
            \end{split}
        \end{equation}
        Let us define the sequence $\{A_t\}_{t\geq 0}$ as follows:
        \begin{equation}
            A_t = 
                \begin{cases}
                    1 , t = 0\\
                    \prod \limits_{i=1}^t (1-\alpha_i),~ t\geq 1.
                \end{cases}
        \end{equation}
        Then, $A_t = (1-\alpha_t)A_{t-1}$. Also, we define $\alpha_0 = 1$.
        Then, dividing both sides of \eqref{eq:conv_to_sum} by $A_t$, we get
        \begin{gather}
            \frac{1}{A_t}\left(f(x_{t+1}) - f(x^{\ast})\right)\leq \frac{1}{A_t} \left(1-\alpha_t)(f(x_t) - f(x^{\ast}) \right) \notag\\
             + \frac{1}{A_t} \left(\frac{\delta_1}{\gamma} + \gamma\delta_1(\alpha_t \aa{D})^2 + \sum \limits_{i = 2}^p \frac{\delta_i }{(i-1)!}(\alpha_t \aa{D})^i +  \frac{L_p + p\sigma }{{(p+1)!}}(\alpha_t \aa{D})^{p+1}\right)\notag\\
            = \frac{1}{A_{t-1}}(f(x_t) - f(x^{\ast})) + \frac{1}{A_{t}}\left( \frac{\delta_1}{\gamma} + \gamma \delta_1 (\alpha_t \aa{D})^2 + \sum \limits_{i = 2}^p \frac{\delta_i }{(i-1)!}(\alpha_t \aa{D})^i +  \frac{L_p + p\sigma }{{(p+1)!}}(\alpha_t \aa{D})^{p+1}\right)\label{eq:summing}.
        \end{gather}
         Summing  both sides of inequality \eqref{eq:summing} for $t = 0, \ldots, T$ , we obtain:
       \begin{gather}
            \frac{1}{A_T}(f(x_{{T}+1}) - f(x^{\ast}))  \leq \frac{(1-\alpha_0)}{A_0}(f(x_0) - f(x^{\ast})) \notag \\
            + \left(\frac{\delta_1}{\gamma} \sum \limits_{t=0}^{T} \frac{ 1}{A_t} + \gamma \delta_1\aa{D}^2 \sum \limits_{t=0}^{T} \frac{ \alpha_t^2}{A_t} + 
             \sum \limits_{i = 2}^p \frac{\delta_i \aa{D}^i}{(i-1)!} {\sum \limits_{t=0}^{T} \frac{ \alpha_t ^{i}}{A_t}} 
             + \frac{(L_p + p\sigma )\aa{D}^{p+1}}{{(p+1)!}}  {\sum \limits_{t=0}^T \frac{ \alpha_t ^{p+1}}{A_t}} \right)  \notag\\
              \stackrel{\alpha_0=1}{=} 
              \frac{\delta_1}{\gamma} \sum \limits_{t=0}^{T} \frac{ 1}{A_t} + \gamma \delta_1\aa{D}^2 \sum \limits_{t=0}^{T} \frac{ \alpha_t^2}{A_t}  + \sum \limits_{i = 2}^p \frac{\delta_i \aa{D}^i}{(i-1)!} {\sum \limits_{t=0}^{T} \frac{ \alpha_t ^{i}}{A_t}} 
             + \frac{(L_p + p\sigma )\aa{D}^{p+1}}{{(p+1)!}}  {\sum \limits_{t=0}^T \frac{ \alpha_t ^{p+1}}{A_t}}.\label{eq:last} 
        \end{gather}
        Let us take 
        \begin{equation}
        \label{alpha_t}
            \alpha_t = \frac{p+1}{t+p+1}, ~ t \geq 1.
        \end{equation}
        Then, we have
        \begin{equation}\label{eq:A_t_bound}
            A_{T} =\prod_{t=1}^{T}\left(1-\alpha_{t}\right)=\prod_{t=1}^{T} \frac{t}{t+p+1}=\frac{T !(p+1) !}{(T+p+1) !}=(p+1) ! \prod_{j=1}^{p+1} \frac{1}{T+j} \leq \frac{(p+1)!}{(T+1)^{p+1}},
         \end{equation}
        which gives, for all $i =0,\ldots,p+1$,
        \begin{equation}
        \label{eq:sum_A_t_bound}
            \begin{aligned}
                \sum_{t=0}^{T} \frac{A_{T} \alpha_{t}^{i}}{A_{t}} &=\sum_{t=0}^{T} \frac{(p+1)^{i}}{(t+p+1)^{i}} \prod_{j=1}^{p+1}\frac{t+j}{T+j} =(p+1)^{i} \prod_{j=1}^{p+1} \frac{1}{T+j} \sum_{t=0}^{T}  \frac{\prod_{j=1}^{p+1} (t+j)}{(t+p+1)^i}.
            \end{aligned}
        \end{equation}
        To estimate the last sum we first prove that the elements of the sum are non-decreasing. Indeed, we have
        \begin{equation*}
            1 \leq 1 +  \frac{1}{t+p+1} \leq  1 + \frac{1}{t+j}, \quad \forall j \in [1, \ldots, p],
        \end{equation*}
        and, hence, for all $i =1,\ldots,p+1$,
        \begin{align*}
        \ls 1 +  \frac{1}{t+p+1}\rs^i \leq \prod_{j=1}^{p+1} \ls 1 + \frac{1}{t+j} \rs \\
        \Leftrightarrow
        \ls\frac{t+p+2}{t+p+1}\rs^i \leq \prod_{j=1}^{p+1} \frac{t+j+1}{t+j}\\
          \Leftrightarrow  \frac{\prod_{j=1}^{p+1} (t+j)}{(t+p+1)^i} \leq \frac{\prod_{j=1}^{p+1} (t+1+j)}{(t+p+2)^i} 
        \end{align*}
 
        Thus, we have shown that the summands in the RHS of \eqref{eq:sum_A_t_bound} are growing, whence we get the next upper bound for the sum, for all $i =0,\ldots,p+1$,
        \begin{equation}
            \label{eq:ATat/At_i0}
            \begin{aligned}
            \sum_{t=0}^{T} \frac{A_{T} \alpha_{t}^{i}}{A_{t}} 
            &=(p+1)^{i} \prod_{j=1}^{p+1} \frac{1}{T+j} \sum_{t=0}^{T}  \frac{\prod_{j=1}^{p+1} (t+j)}{(t+p+1)^i}\\
            &\leq (p+1)^{i} \prod_{j=1}^{p+1} \frac{1}{T+j}\cdot (T+1) \cdot  \frac{ \prod_{j=1}^{p+1} (T+j)}{(T+p+1)^i}
            \leq\frac{(T+1)(p+1)^{i}}{(T+p+1)^{i}}.
            \end{aligned}
        \end{equation}
        
        Additionally, for all $i = 1,\ldots,p+1$,
        \begin{equation}
        \label{eq:ATat/At}
            \sum_{t=0}^{T} \frac{A_{T} \alpha_{t}^{i}}{A_{t}} \leq\frac{(T+1)(p+1)^{i}}{(T+p+1)^{i}} \leq\frac{(p+1)^{i}}{(T+p+1)^{i-1}}
        \end{equation}
        
        From \eqref{eq:last} we get the statement of the Theorem:
        \begin{align*}
            f(x_{{T}+1}) - f(x^{\ast}) &\leq  \frac{\delta_1}{\gamma}\sum \limits_{t=0}^{T}\frac{A_T}{A_t} + \gamma\delta_1 \aa{D}^2\sum \limits_{t=0}^{T}\frac{A_T \alpha_t^2}{A_t} \\
            &+ \sum \limits_{i = 2}^p \frac{\delta_i \aa{D}^i}{(i-1)!} {\sum \limits_{t=1}^{T} \frac{A_T \alpha_t ^{i}}{A_t}} 
             + \frac{(L_p + p\sigma )\aa{D}^{p+1}}{{(p+1)!}}  {\sum \limits_{t=1}^T \frac{A_T \alpha_t ^{p+1}}{A_t}}\\
             &\stackrel{\eqref{eq:ATat/At_i0},\eqref{eq:ATat/At}}{\leq}
             \frac{\delta_1}{\gamma}(T + 1) + \frac{\gamma\delta_1\aa{D}^2(p+1)^2}{T+p+1} \\
             &+ \sum \limits_{i = 2}^p \frac{\delta_i \aa{D}^i}{(i-1)!} \frac{(p+1)^{i}}{({T}+p+1)^{i-1}} 
             + \frac{(L_p + p\sigma )\aa{D}^{p+1}}{{(p+1)!}} \frac{(p+1)^{p + 1}}{({T}+p+1)^{p}}.
        \end{align*}
        
    \aa{Next, one can set 
        $$\gamma = \frac{\sqrt{(T+1)(T+p+1)}}{D(p+1)} = O\ls\frac{T}{D}\rs.$$
        Then we obtain the following convergence rate
        \begin{gather*}
            f(x_{{T}+1}) - f(x^{\ast}) \leq 
            2(p+1)\delta_1D 
            + \sum \limits_{i = 2}^p \frac{\delta_i \aa{D}^i}{(i-1)!} \frac{(p+1)^{i}}{({T}+p+1)^{i-1}} 
             + \frac{(L_p + p\sigma )\aa{D}^{p+1}}{{(p+1)!}} \frac{(p+1)^{p + 1}}{({T}+p+1)^{p}} \\
             = \sum \limits_{i = 1}^p O\ls \frac{\delta_i \aa{D}^i}{T^{i-1}} \rs
             + O\ls\frac{L_p \aa{D}^{p+1}}{T^{p}}\rs.
        \end{gather*}
    }
\end{proof}

This result provides an upper bound for the objective residual after $T$ iterations of the Inexact Tensor Method. The last term in the RHS of \eqref{eq:convergence} and \eqref{eq:convergence_O} corresponds to the case of exact Tensor method, i.e. $\delta_i=0$, $i=1,...,p$, and provides a similar convergence rate to that of non-accelerated Tensor method in \cite{nesterov2019implementable}. The other terms in the RHS show how inexactness in each derivative influences the convergence rate. 
In particular, if, for some $i$-th derivative, the corresponding error $ \delta_i$ is non-zero, the convergence rate becomes worse since the convergence rate bound has term proportional to $\delta_i/T^{i-1}$. Importantly, when the error $\delta_1$ in the gradient is non-zero, our bound does not guarantee decrease of the objective below the error $O(\delta_1D)$. The same effect can be observed for first-order methods \cite{devolder2013exactness}. We now give sufficient conditions for the errors $\delta_i$, $i=1,\ldots,p$ such that the method is still guaranteed to find an $\e$-solution.
In particular, this result  answers the following question. Assume that the errors in the derivatives can be controlled and made as small as we would like them to be. Then, how small should we take the error in each derivative if we would like to achieve objective residual smaller than $\e$? As we will see below in Section \ref{sec:smpl}, for  stochastic optimization problems, this control over the errors can be achieved by increasing the sample size for each derivative.

\begin{corollary}\label{lm:inexactenss_nonacc}
Let assumptions of Theorem \ref{thm:convergence} hold and let $\e > 0$ be the desired solution accuracy.
Further, let the levels of inexactness in Condition \ref{cnd:sampling} satisfy the following inequalities:
\begin{equation}
    \label{eq_delta_1}
    \delta_1 \leq \frac{1}{2(p+1)^2}\frac{\e}{\aa{D}},
\end{equation}
\begin{equation}
\label{eq_delta_i}
    \delta_i \leq  \frac{(i-1)!}{(p+1)}\left(\frac{(L_p + p\sigma )}{p!}\right)^{\frac{i-1}{p}}\ls\frac{\e}{\aa{D}(p+1)}\rs^{\frac{p-i+1}{p}}, \quad \forall i \in [2, \ldots, p].
\end{equation}
Also, let the number of iterations $T$ of Algorithm \ref{alg:inexact} satisfy 
\begin{equation}
    \label{eq_T_bounds}
    T = \max \lb 1; \ls \frac{(p+1)^{p+1}}{{p!}}\frac{(L_p + p\sigma )\aa{D}^{p+1}}{\e}\rs^{\frac{1}{p}} - p - 1\rb,
\end{equation}
and $\gamma$ be chosen as
\begin{equation}
\label{gamma}
    \gamma = \frac{T+p+1}{\aa{D}(p+1)}=\ls \frac{p+1}{{p!}}\frac{(L_p + p\sigma )\aa{D}}{\e}\rs^{\frac{1}{p}}.
\end{equation}
Then $x_{T+1}$ is an $\e$-solution of problem \eqref{eq:PrSt}, i.e. $f(x_{T+1})-f(x^{\ast})\leq \e$. 
\end{corollary}

\begin{proof}
    First, by taking $\gamma$ from \eqref{gamma} we get
    \begin{align*}
    \frac{\delta_1}{\gamma}(T + 1) +
        \frac{\gamma\delta_1\aa{D}^2(p+1)^2}{T+p+1}\leq 2\delta_1  \aa{D}(p+1)\leq \frac{\e}{p+1}
    \end{align*}
    By our choice of  $\gamma, ~ T, ~ \delta_i, ~ i=1, \ldots, p$ we have:
    \begin{gather*}
        \frac{\delta_i \aa{D}^i}{(i-1)!} \frac{(p+1)^{i}}{({T}+p+1)^{i-1}} \leq \frac{\e}{p+1}~ i=2, \ldots, p\\
        \frac{(L_p + p\sigma )\aa{D}^{p+1}}{{(p+1)!}} \frac{(p+1)^{p + 1}}{({T}+p+1)^{p}} = \frac{\e}{p+1}.
    \end{gather*}
    Then, from \eqref{eq:convergence} we have:
    \begin{equation*}
        f(x_{T+1}) - f(x^{\ast}) \leq \e.
    \end{equation*}
\end{proof}

Thus, we showed that the number of steps sufficient for the  Inexact Tensor Method   to find an $\e$-solution to problem \eqref{eq:PrSt} is the same as for non-accelerated exact tensor method \cite{nesterov2019implementable} and is equal to
$
    O\left({\left(\dfrac{L_p \aa{D}^{p+1}}{\e}\right)^{1/p}}\right)
$ under the assumption that the inexactness $\delta_i$ of each derivative satisfies $\delta_i = O \ls\e^{\frac{p-i+1}{p}}\rs$, $i=1,\ldots,p$. Interestingly, as the order of the derivative increases, the requirement for the accuracy of the derivative becomes less strict.

\aa{
\begin{remark}
    In the statement of Theorem \ref{thm:convergence} we defined $D = \max\limits_{x\in L^\prime(x_0)} \|x-x_0\|$, where $\mathcal{L}^\prime(x_0) = \lb x\vert f(x) \leq f(x_0) + \frac{T\delta_1}{2\gamma}\rb.$
    Choosing the inexactness level $\delta_1$ and the number of iterations $T$ as in Corollary \ref{lm:inexactenss_nonacc}, we have  $\dfrac{T\delta_1}{2\gamma} \leq \dfrac{\e}{2(p+1)}.$
    Therefore, the Inexact Tensor Method is monotone up to a small inaccuracy $\frac{\e}{2(p+1)}$.
 \end{remark}
}

\section{Accelerated Inexact Tensor Method}\label{sec:acceleration}
In this section, we use the estimating sequences technique to propose
Accelerated Inexact Tensor Method and prove its convergence rate theorem providing faster rates than that of Algorithm \ref{alg:inexact}.

\begin{algorithm}
  \caption{Accelerated Inexact Tensor Method}\label{alg:inexact_acc}
  \begin{algorithmic}[1]
      \STATE \textbf{Input:} convex function $f$ such that $\nabla^p f$ is  $L_p$-Lipschitz; $x_0$ is starting point; constants \aa{$\sigma \geq \max\lb 1; \tfrac{3}{p}\rb L_p
      $}, \aa{$\eta_i\geq 3$ for $i = 2, \ldots, p$}; nonnegative nondecreasing sequences $\{\bk^t_i\}_{t \geq 0}$ for $i = 2, \ldots, p+1$, and
      \begin{equation}\label{eq:alphas}
          \alpha_t = \frac{p+1}{t + p + 1}, ~~~ A_t = \prod \limits_{j=1}^t(1 -\alpha_j), ~~~ A_0 = 1.
      \end{equation}
      \STATE \textbf{Precomputation:}
    Call the inexact oracle to compute ${G}_{x_0, i}$ for $i = 1, \ldots, p$ such that Condition $\ref{cnd:sampling}$ is satisfied. Compute 
    \begin{equation}\label{eq:x1_acc}
        x_1=\arg \min _{x \in \mathbb{R}^{n}} \{\phi_{x_0, p}(x)+ \sum \limits_{i = 2}^p \frac{\delta_i}{(i - 1)!}d_i(x - x_0) + \frac{\sigma}{(p-1)!}d_{p+1}(x - x_0)\}
    \end{equation}
    \begin{equation}\label{eq:psi_1}
        y_1= \arg \min _{x \in \mathbb{R}^{n}}\left\{\psi_{1}(x):=f\left(x_{1}\right)+ \sum \limits_{i = 2}^{p+1} \frac{\bk_i^{0}}{(i - 1)!}d_i(x - x_0) \right\}.
    \end{equation}
    \FOR{$t \geq 1$} 
        \STATE Set 
            \begin{equation}\label{eq:u_t}
                v_t = (1 - \alpha_t)x_t + \alpha_t y_t,
            \end{equation}
        \STATE Call the inexact oracle to compute ${G}_{v_t, i}$ for $i = 1, \ldots, p$ such that Condition $\ref{cnd:sampling}$
        is satisfied and set
            \begin{equation}\label{eq:acc_prox_x_step}
                x_{t+1}=\arg \min _{x \in \mathbb{R}^{n}} \{\phi_{v_t, p}(x) + \sum \limits_{i = 2}^p \frac{\eta_i\delta_i}{(i - 1)!}d_i(x - v_t) + \frac{\sigma}{(p-1)!}d_{p+1}(x - v_t)\}.
            \end{equation}
        \STATE Compute 
            \begin{equation}\label{eq:estimating_seq}
            \begin{aligned}
                y_{t+1}=\arg \min _{x \in \mathbb{R}^{n}}\left\{\psi_{t+1}(x):=\psi_{t}(x)+ \sum \limits_{i = 2}^{p+1} \frac{\bk^{t}_i - \bk^{t-1}_i}{(i-1)!} d_{i}(x - x_0)+\frac{\alpha_{t}}{A_{t}} 
                \phi_{x_{t+1}, 1}(x) \right\}.
                \end{aligned}
            \end{equation}
    \ENDFOR
  \end{algorithmic}
\end{algorithm}

In Algorithm \ref{alg:inexact_acc} we utilize the estimating sequence technique to obtain acceleration. In this technique, to prove the convergence theorem (Theorem \ref{thm:acc_convergence}) one needs to construct a lower and an upper bound for the estimating sequence $\psi_t(x)$ based on the objective function. So, the full proof is organized as follows:
\begin{itemize}
    \item Lemma \ref{lem:upper_seq} provides an upper bound for the estimating sequence  $\psi_t(x)$ \eqref{eq:psi_1}, \eqref{eq:estimating_seq}; 
    \item Lemma \ref{lem:step} provides a lower bound on $\psi_t(x)$ based on results of technical Lemmas
    \ref{lem:dual}- \ref{lm:argmin}; 
    \item finally, everything is combined together in Theorem \ref{thm:acc_convergence} in order to prove convergence and obtain convergence rate.
\end{itemize}

The following lemma shows that the sequence of functions $\psi_t(x)$ generated by Algorithm \ref{alg:inexact_acc} can be upper bounded by the properly regularized objective function.

\begin{lemma}\label{lem:upper_seq}
    If Condition $\ref{cnd:sampling}$ is satisfied and $\psi_t(x)$ is defined as in $\eqref{eq:psi_1}$, $\eqref{eq:estimating_seq}$ then, for any $x \in \mathbb{R}^n$, we have
    \begin{gather*}
        \psi_t(x) \leq \frac{f(x)}{A_{t - 1}} +\delta_1\|x_1-x_0\|+ \aa{\delta_1\|x-x_0\|} + \sum \limits_{i = 2}^p \frac{ 2\delta_i + \bk^{t - 1}_i}{(i - 1)!}d_i(x - x_0) \\+ \frac{L_p + p\sigma + \bk^{\aa{t-1}}_{p+1}}{p!}d_{p+1}(x - x_0)
        + \sum \limits_{i = 1}^{t - 1} \frac{\alpha_i}{A_i} \langle G_{x_{i+1}, 1} - \nabla f(x_{i+1}), x - x_{i+1}\rangle.
    \end{gather*}
\end{lemma}

\begin{proof}
    \begin{align*}
        f(x_1)  &\stackrel{\eqref{eq:func_bnd}}{\leq} \phi_{x_0, p}(x_1)+ \delta_1 \|x_1-x_0\| + \sum \limits_{i = 2}^p \frac{\delta_i}{(i - 1)!}d_i(x_1 - x_0) + \frac{L_p}{p!}d_{p+1}(x_1 - x_0)\\
         &\stackrel{L_p \leq \sigma}{\leq} \phi_{x_0, p}(x_1) + \delta_1 \|x_1-x_0\| + \sum \limits_{i = 2}^p \frac{\delta_i}{(i - 1)!}d_i(x_1 - x_0) + \frac{\sigma}{p!}d_{p+1}(x_1 - x_0)\\
        &\stackrel{\sigma \leq p\sigma}{\leq} \phi_{x_0, p}(x_1) + \delta_1 \|x_1-x_0\| + \sum \limits_{i = 2}^p \frac{\delta_i}{(i - 1)!}d_i(x_1 - x_0) + \frac{p\sigma}{p!}d_{p+1}(x_1 - x_0)\\
        &\stackrel{\eqref{eq:x1_acc}}{\leq} \phi_{x_0, p}(x) + \delta_1 \|x_1-x_0\|  + \sum \limits_{i = 2}^p \frac{\delta_i}{(i - 1)!}d_i(x - x_0) + \frac{p\sigma}{p!}d_{p+1}(x - x_0)\\
        &\stackrel{\eqref{eq:func_bnd}}{\leq}
        f(x) + \delta_1 \|x_1-x_0\| + \aa{\delta_1\|x-x_0\|}  + \sum \limits_{i = 2}^p \frac{2\delta_i}{(i - 1)!}d_i(x - x_0) + \frac{L_p + p\sigma}{p!}d_{p+1}(x - x_0).
    \end{align*}
    Therefore, 
    \begin{gather} 
        \psi_1(x) = f(x_1) + \sum \limits_{i = 2}^{p+1} \frac{\bk^0_i}{(i - 1)!}d_i(x - x_0) \label{eq:lem_upper_seq_pr1} \leq f(x) + \delta_1 \|x_1-x_0\| \\
        + \aa{\delta_1\|x-x_0\|}  +  \sum \limits_{i = 2}^p \frac{ 2\delta_i + \bk^0_i}{(i - 1)!}d_i(x - x_0) + \frac{L_p + p\sigma + p\bk^{0}_{p+1}}{p!}d_{p+1}(x - x_0). \notag
    \end{gather}
    From \eqref{eq:estimating_seq} we have
    \begin{gather}
        \psi_t(x) = \psi_1(x) + \sum \limits_{i = 2}^{p+1} \frac{\bk^{t - 1}_i - \bk^{0}_i}{(i - 1)!}d_i(x - x_0) + \sum \limits_{j = 1}^{t - 1} \frac{\alpha_j}{A_j} \phi_{x_{j+1}, 1}(x) \notag\\
        = \psi_1(x) + \sum \limits_{i = 2}^{p+1} \frac{\bk^{t - 1}_i - \bk^{0}_i}{(i - 1)!}d_i(x - x_0) + \sum \limits_{j = 1}^{t - 1} \frac{\alpha_j}{A_j} \Phi_{x_{j+1}, 1}(x)\notag \\
        + \sum \limits_{j = 1}^{t - 1} \frac{\alpha_j}{A_j} \la G_{x_{j+1}, 1} - \nabla f(x_{j+1}), x - x_{j+1}\ra.
        \label{eq:lem_upper_seq_pr2}
    \end{gather}
From \eqref{eq:alphas} we have that, for all $j \geq 1$, $A_j = A_{j-1}(1-\alpha_j)$, which leads to $\frac{\alpha_j}{A_j}=\frac{1}{A_j}-\frac{1}{A_{j-1}}$. Hence, we have $\sum \limits_{j = 1}^{t - 1} \frac{\alpha_j}{A_j} = \frac{1}{A_{t-1}} - \frac{1}{A_0} $ and, using the convexity of the objective $f$, we get
    \begin{gather}
        \sum \limits_{j = 1}^{t - 1} \frac{\alpha_j}{A_j} \Phi_{x_{j+1}, 1}(x) \leq f(x)\sum \limits_{j = 1}^{t - 1} \frac{\alpha_j}{A_j} = f(x)\left(\frac{1}{A_{t-1}} - \frac{1}{A_0}\right) . \label{eq:lem_upper_seq_pr3}
    \end{gather}
    Finally, combining all the inequalities from above and using that $A_0=1$, we obtain 
    \begin{gather}
        \psi_t(x)  \stackrel{\eqref{eq:lem_upper_seq_pr2},\eqref{eq:lem_upper_seq_pr3}}{\leq} \psi_1(x) + \sum \limits_{i = 2}^{p+1} \frac{\bk^{t - 1}_i - \bk^{0}_i}{(i - 1)!}d_i(x - x_0) + f(x)\left(\frac{1}{A_{t-1}} - \frac{1}{A_0}\right) \notag \\
        + \sum \limits_{j = 1}^{t - 1} \frac{\alpha_j}{A_j} \la G_{x_{j+1}, 1} - \nabla f(x_{j+1}), x - x_{j+1}\ra \notag \\
        \stackrel{\eqref{eq:lem_upper_seq_pr1},A_0=1}{\leq} \frac{f(x)}{A_{t - 1}} + \delta_1 \|x_1-x_0\|+ \aa{\delta_1\|x-x_0\|} + \sum \limits_{i = 2}^p \frac{2\delta_i + \bk^{t - 1}_i}{(i - 1)!}d_i(x - x_0)\notag \\ 
        + \frac{L_p + p\sigma + p\bk^{t}_{p+1}}{p!}d_{p+1}(x - x_0)
        + \sum \limits_{j = 1}^{t - 1} \frac{\alpha_j}{A_j} \la G_{x_{j+1}, 1} - \nabla f(x_{j+1}), x - x_{j+1}\ra.
    \end{gather}
\end{proof}

The next Lemma characterizes the progress of the tensor step  \eqref{eq:acc_prox_x_step} in Algorithm $\ref{alg:inexact_acc}$.

\begin{lemma}\label{lem:scalar_lb_cases}
    Let $\{x_t, v_t\}_{t \geq 1}$ be generated by Algorithm $\ref{alg:inexact_acc}$. Then, the following holds.
    \begin{itemize}
        \item If 
        \begin{equation}
            \label{delta_1_case}
            \delta_1 \geq \sum \limits_{i = 2}^p \frac{\delta_i}{(i-1)!}\|x_{t+1} - v_t\|^{i-1} + \frac{L_p}{p!}\|x_{t+1} - v_t\|^{p},
        \end{equation}
        then
        $$f(x_{t+1}) - f(x^{\ast}) \leq \ls\max\lb\frac{p\sigma}{L_p};\max_{i}\lb\eta_i\rb\rb + 2\rs\delta_1\|x_{t+1} - x^{\ast}\|;$$
        \item else
        \begin{equation}
        \begin{gathered}
        \label{eq:scalar_lb_cases}
        \langle \nabla f(x_{t+1}), v_t - x_{t+1} \rangle 
        \geq 
        \\ \frac{\sqrt{5}}{3} \min \left\{\min \limits_{i = 2, \ldots, p} \|\nabla f(x_{t+1})\|^\frac{i}{i-1}\left( \frac{(i-1)!}{p(\eta_i + 2)\delta_i }\right)^\frac{1}{i-1}, \right.
        \left.\|\nabla f(x_{t+1})\|^\frac{p+1}{p}
        \left( \frac{(p-1)!}{2L_p + p\sigma}\right)^\frac{1}{p}\right\}.
         \end{gathered}
    \end{equation}
    \end{itemize}
\end{lemma}
\begin{proof}
    For simplicity, we denote
    \begin{equation}
        \label{big_sum_lemma6}
        \zeta_{t+1} = \sum \limits_{i = 2}^p \frac{\eta_i\delta_i}{(i - 1)!}\|x_{t+1} - v_t\|^{i - 2} + \frac{\sigma}{(p-1)!}\|x_{t+1} - v_t\|^{p-1}.
    \end{equation}
    By the optimality condition in \eqref{eq:acc_prox_x_step} 
    
    \begin{equation}
   \label{eq:opt_cnd_acc}
    \begin{gathered}
    0 = \nabla \phi_{v_t, p}(x_{t+1})+\sum \limits_{i = 2}^p \frac{\eta_i\delta_i}{(i - 1)!}\nabla d_i(x_{t+1} - v_t)  + \frac{\sigma}{(p - 1)!}\nabla d_{p+1}(x_{t+1} - v_t) \\ \stackrel{\eqref{big_sum_lemma6}}{=} \nabla \phi_{v_t, p}(x_{t+1}) + \zeta_{t+1} (x_{t+1} - v_t).
    \end{gathered}
    \end{equation}
    
    From the optimality condition \eqref{eq:opt_cnd_acc} we obtain 
    \begin{equation}
    \label{grad_delta_1}
    \begin{gathered}
        \|\nabla f(x_{t+1})\| =  \left\|\nabla f(x_{t+1}) - \nabla \phi_{v_t, p}(x_{t+1}) -\zeta_{t+1} (x_{t+1} - v_t)\right\|\\
        \leq \|\nabla \phi_{v_t, p}(x_{t+1}) - \nabla f(x_{t+1})\| + \zeta_{t+1} \|x_{t+1} - v_t\| \\
        \stackrel{\eqref{eq:1deriv_bnd}}{\leq} \delta_1 + 
        \sum \limits_{i = 2}^p \frac{(\eta_i+1)\delta_i}{(i-1)!}\|x_{t+1} - v_t\|^{i - 1} + \frac{L_p + p\sigma}{p!}\|x_{t+1} - v_t\|^p.
        \end{gathered}
\end{equation}
    If $\delta_1$ is dominating, i.e. \eqref{delta_1_case} holds, then
    \begin{align*}
        \|\nabla f(x_{t+1})\| 
        &\leq \delta_1 + 
        \sum \limits_{i = 2}^p \frac{(\eta_i+1)\delta_i}{(i-1)!}\|x_{t+1} - v_t\|^{i - 1} + \frac{L_p + p\sigma}{p!}\|x_{t+1} - v_t\|^p\\
        &\stackrel{\eqref{delta_1_case}}{\leq} \ls\max\lb\frac{p\sigma}{L_p};\max_{\aa{i \in \lb 2,\ldots, p\rb}}\lb\eta_i\rb\rb + 2\rs\delta_1.
\end{align*}
    Since $f$ is convex, we have
    \begin{align*}
    f(x_{t+1}) - f(x^{\ast}) &\leq \langle \nabla f(x_{t+1}), x_{t+1} - x^{\ast} \rangle \leq \|\nabla f(x_{t+1})\|\|x_{t+1} - x^{\ast}\|    \\ 
    &\leq \ls\max\lb\frac{p\sigma}{L_p};\max_{\aa{i \in \lb 2,\ldots, p\rb}}\lb\eta_i\rb\rb + 2\rs\delta_1\|x_{t+1} - x^{\ast}\|.
    \end{align*}
    This proves the first statement. 
    
    Now, we move to the second statement and for the rest of the proof we assume that
    \begin{equation}
        \label{delta_1_not}
        \delta_1 \leq \sum \limits_{i = 2}^p \frac{\delta_i}{(i-1)!}\|x_{t+1} - v_t\|^{i-1} + \frac{L_p}{p!}\|x_{t+1} - v_t\|^{p}.
    \end{equation}
    We start with getting an upper bound for $\|\nabla \phi_{v_t, p}(x_{t+1}) - \nabla f(x_{t+1})\|$.
    Combining the above inequality with \eqref{eq:1deriv_bnd} we have 
    \begin{gather*}
        \|\nabla \phi_{v_t, p}(x_{t+1}) - \nabla f(x_{t+1})\| 
        \stackrel{\eqref{eq:1deriv_bnd}}{\leq} \delta_1+
         \sum \limits_{i = 2}^p \frac{\delta_i}{(i-1)!}\|x_{t+1} - v_t\|^{i-1} + \frac{L_p}{p!}\|x_{t+1} - v_t\|^{p}\\
         \stackrel{\eqref{delta_1_not}}{\leq} 
         \sum \limits_{i = 2}^p \frac{2\delta_i}{(i-1)!}\|x_{t+1} - v_t\|^{i-1} + \frac{2L_p}{p!}\|x_{t+1} - v_t\|^{p}\\
         \stackrel{\eta_i\geq 3}{\leq} 
         \frac{2}{3}\ls \sum \limits_{i = 2}^p \frac{\eta_i\delta_i}{(i-1)!}\|x_{t+1} - v_t\|^{i-1}\rs + \frac{2L_p}{p!}\|x_{t+1} - v_t\|^{p}\\
         \stackrel{3L_p\leq p\sigma}{\leq} 
         \frac{2}{3}\ls \sum \limits_{i = 2}^p \frac{\eta_i\delta_i}{(i-1)!}\|x_{t+1} - v_t\|^{i-1} + \frac{\sigma}{(p-1)!}\|x_{t+1} - v_t\|^{p}\rs\\
         \stackrel{\eqref{big_sum_lemma6}}{=} \frac{2}{3}\zeta_{t+1}\|x_{t+1} - v_t\|.
    \end{gather*}
   Next, from the previous inequality and optimality condition \eqref{eq:opt_cnd_acc}, we get
    \begin{gather*}
        \frac{4}{9}\zeta_{t+1}^2\|x_{t+1} - v_t\|^2 \geq \|\nabla \phi_{v_t, p}(x_{t+1}) - \nabla f(x_{t+1})\|^2 
        \stackrel{\eqref{eq:opt_cnd_acc}}{=} \left\|\nabla  f(x_{t+1}) +  \zeta_{t+1} (x_{t+1} - v_t) \right\|^2 \\
        = 2 \langle \nabla f(x_{t+1}), x_{t+1} - v_t \rangle \zeta_{t+1}
        + \zeta_{t+1}^2\|x_{t+1} - v_t\|^2 + \|\nabla f(x_{t+1}) \|^2 .
    \end{gather*}
    Hence, 
    \begin{gather*}
        2 \langle \nabla f(x_{t+1}), v_t - x_{t+1} \rangle \zeta_{t+1} 
        \geq
        \|\nabla f(x_{t+1}) \|^2 + \frac{5}{9}\zeta_{t+1}^2\|x_{t+1} - v_t\|^2\\
        \geq \frac{2\sqrt{5}}{3}\zeta_{t+1}\|\nabla f(x_{t+1}) \| \|x_{t+1} - v_t\|.
    \end{gather*}
    Dividing both sides by $2\zeta_{t+1}$, we finally get
    \begin{equation}\label{eq:scalar_lb}
        \la\nabla f(x_{t+1}), v_t - x_{t+1} \ra 
        \geq
        \frac{\sqrt{5}}{3}\|\nabla f(x_{t+1}) \| \|x_{t+1} - v_t\|.
    \end{equation}
To bound $\|x_{t+1}-v_t\|$ from below using the gradient norm $\|\nabla f(x_{t+1})\|$, we use inequality \eqref{grad_delta_1}:
\begin{gather*}
     \|\nabla f(x_{t+1})\| \leq \delta_1 + 
        \sum \limits_{i = 2}^p \frac{(\eta_i+1)\delta_i}{(i-1)!}\|x_{t+1} - v_t\|^{i - 1} + \frac{L_p + p\sigma}{p!}\|x_{t+1} - v_t\|^p\\
        \stackrel{\eqref{delta_1_not}}{\leq}
        \sum \limits_{i = 2}^p \frac{(\eta_i+2)\delta_i}{(i-1)!}\|x_{t+1} - v_t\|^{i - 1} + \frac{2L_p + p\sigma}{p!}\|x_{t+1} - v_t\|^p.
\end{gather*}

Next, we consider $p-1$ cases depending on which term dominates in the RHS of the last inequality. 

\begin{itemize}
    \item If, for some $i = 2, \ldots, p$, the term
    $\frac{(\eta_i + 2)\delta_i}{(i-1)!}\|x_{t+1} - v_t\|^{i-1}$ dominates the others, then we get the following bound
    $$\|x_{t+1} - v_t\| \geq \left( \frac{\|\nabla f(x_{t+1})\| (i-1)!}{p(\eta_i +2)\delta_i} \right)^{\frac{1}{i-1}}.$$
    
    \item If the term $\frac{2L_p + p\sigma}{p!}\|x_{t+1} - v_t\|^{p}$ dominates the others, then
    $$\|x_{t+1} - v_t\| \geq \left(\frac{\|\nabla f(x_{t+1})\| (p-1)!}{2L_p + p\sigma} \right)^{\frac{1}{p}}.$$
\end{itemize} 
Clearly, $\|x_{t+1} - v_t\|$ is bounded from below by the minimum of these $p$ lower bounds. Combining this minimum bound with the bound \eqref{eq:scalar_lb} we finish the proof.
\end{proof}

We will also use the next technical lemma \cite{nesterov2008accelerating,ghadimi2017secondorder} on Fenchel conjugate for the $p$-th power of the norm.
\begin{lemma}\label{lem:dual}
    Let $g(z)=\frac{\theta}{p}\|z\|^{p}$ for $p \geq 2$ and $g^{*}$ be its conjugate function i.e., $g^{*}(v)=\sup _{z}\{\langle v, z\rangle-$ $g(z)\} .$ Then, we have
$$
g^{*}(v)=\frac{p-1}{p}\left(\frac{\|v\|^{p}}{\theta}\right)^{\frac{1}{p-1}}
$$
Moreover, for any $v, z \in \mathbb{R}^{n}$, we have $g(z)+g^{*}(v)-\langle z, v\rangle \geq 0 .$
\end{lemma}

The next step is to provide a lower bound $\psi_t(x) \geq \psi_t^* := \min_x \psi_t(x) \geq \frac{f(x_t)}{A_{t-1}} \aa{+ err_t}$ for all $x$, where $err_t$ is some error term that will be defined later. The convergence rate of Algorithm \ref{alg:inexact_acc} will follow from this bound and Lemma \ref{lem:upper_seq}. The proof of the desired lower bound is quite technical and requires several auxiliary lemmas. After that we combine all the technical results together to obtain convergence rate in the proof of Theorem \ref{thm:acc_convergence}. We start the technical derivations with the following result.

\begin{lemma}\label{lm:argmin}
    Let $h(x)$ be a convex function, $x_0 \in \mathbb{R}^n$, $\theta_i \geq 0$ for $i = 2,\ldots, p + 1$ and 
    $$\bar{x} = \arg \min \limits_{x \in \mathbb{R}^n} \{\bar{h}(x) = h(x) + \sum \limits_{i = 2}^{p + 1}  \theta_i d_i(x -x_0)\}.$$ Then, for all $x\in \mathbb{R}^n$,
    
    $$\bar{h}(x) \geq \bar{h}(\bar{x}) + \sum  \limits_{i = 2}^{p + 1} \left(\frac{1}{2}\right)^{i - 2} \theta_i d_i(x - \bar{x}).$$
\end{lemma}
\begin{proof}
    In \cite{nesterov2008accelerating} it is shown that, for all $x, y \in \mathbb{R}^n$ and for any $i \geq 2$, 
    \begin{equation*}
        d_i(x) - d_i(y) - \langle \nabla d_i(y), x - y \rangle \geq \left(\frac 12 \right)^{i-2}d_i(x - y).
    \end{equation*}
    Using the convexity of $h$, we have
    \begin{gather*}
        \bar{h}(x) = h(x) + \sum \limits_{i = 2}^{p + 1} \theta_i d_i(x - x_0) \geq h(\bar{x}) + \langle\nabla h(\bar{x}), x - \bar{x}\rangle + \sum \limits_{i = 2}^{p+1} \theta_i d_i(x - x_0) \\
        \geq h(\bar{x}) + \langle \nabla h(\bar x), x - \bar x\rangle + \sum \limits_{i = 2}^{p + 1} \theta_i\left( d_i(\bar{x} - x_0) + \langle\nabla d_i(\bar{x} - x_0), x - \bar{x} \rangle + \left(\frac{1}{2}\right)^{i - 2}d_i(x - \bar x)\right)\\
        = \bar h(\bar x) + \langle \nabla \bar h (\bar x), x - \bar x\rangle + \sum \limits_{i = 2}^{p + 1} \left(\frac{1}{2}\right)^{i - 2}\theta_i d_i(x - \bar x) \geq  \bar h(\bar x) +  \sum \limits_{i = 2}^{p + 1} \left(\frac{1}{2}\right)^{i - 2}\theta_i d_i(x - \bar x),
    \end{gather*}
    where the last inequality holds by optimality condition since $\bar h(x)$ is convex.
\end{proof}

Finally, the last technical step is the next Lemma which will be a part of the induction step to prove that $\frac{f(x_t)}{A_{t-1}} + err_t \leq  \min \limits_x \psi_t(x) = \psi^{\ast}_t$.

\begin{lemma}\label{lem:step}
    Let $\{x_t, \aa{y_t}\}_{t \geq 1}$ be generated by Algorithm $\ref{alg:inexact_acc}$ and define 
    \begin{equation}\label{eq:err_seq}
        err_t := \sum \limits_{j=1}^{t-1}\frac{\alpha_j}{A_j}\langle G_{x_{j+1}, 1} - \nabla f(x_{j+1}), y_{j+1} - x_{j+1}\rangle  .
    \end{equation}
    Assume also that
    \begin{equation}\label{eq:lemma_ass}
    \psi^{\ast}_t := \min \limits_x \psi_t(x)   \geq \frac{f(x_t)}{A_{t-1}} + err_t.
    \end{equation}
    Then,
    \begin{gather*}
        \psi_{t+1}^{\ast} \geq \frac{f(x_{t+1})}{A_t} + \frac{1}{A_t}\langle \nabla f(x_{t+1}), v_t - x_{t+1}\rangle + \sum \limits_{i = 2}^{p+1} \left(\frac{1}{2} \right)^{i - 2} \frac{ \bk_i^{t}}{(i - 1)!}d_i(y_{t+1} - y_t) \\
        + \frac{\alpha_t}{A_t}\langle \nabla f(x_{t+1}), y_{t+1} - y_t \rangle 
        + \dk{err_{t+1}} .
    \end{gather*}
\end{lemma}

\begin{proof}

    By definition,
    \begin{gather*}
        \psi_t(x) = f(x_1) + \sum \limits_{i=2}^{p+1} \frac{ \bk^{t-1}_i}{(i-1)!}d_i(x - x_0) + \sum \limits_{j = 1}^{t - 1} \frac{\alpha_j}{A_j}\phi_{x_{j+1}, 1}(x).
    \end{gather*}
    Next, we apply Lemma \ref{lm:argmin} with the following choice of parameters: \\$h(x) = f(x_1) +  \sum \limits_{j = 1}^{t - 1}\frac{\alpha_j}{A_j}\phi_{x_{j+1}, 1}(x)$, $\theta_i = \frac{ \bk^{t-1}_i}{(i-1)!}$ for $i = 2, \ldots, p+1$. \\
    By \eqref{eq:estimating_seq}, $y_t = \argmin \limits_{x\in \mathbb{R}^n} \bar{h}(x)$, and we have
    \begin{gather*}
        \psi_t(x) \geq \psi_t^{\ast} + \sum \limits_{i=2}^{p+1} \left(\frac{1}{2}\right)^{i-2}\frac{ \bk^{t-1}_i}{(i-1)!}d_i(x - y_t) 
        \stackrel{\eqref{eq:lemma_ass}}{\geq} \frac{f(x_t)}{A_{t - 1}} + \sum \limits_{i=2}^{p+1} \left(\frac{1}{2}\right)^{i-2}\frac{ \bk^{t-1}_i}{(i-1)!}d_i(x - y_t)
        +err_t,
    \end{gather*}
    where the last inequality follows from the assumption of the lemma.
    
    By the definition of $\psi_{t+1}(x)$, the above inequality, and convexity of $f$, we obtain
    \begin{gather*}
        \psi_{t+1}(x) =  \psi_t(x) + \sum \limits_{i = 2}^{p+1} \frac{\bk_i^ 
        t - \bk_i^{t-1}}{(i - 1)!}d_i(x - x_0) + \frac{\alpha_t}{A_t}\phi_{x_{t+1}, 1}(x)\\
        \geq \frac{f(x_t)}{A_{t - 1}} + \sum \limits_{i=2}^{p+1} \left(\frac{1}{2}\right)^{i-2}\frac{ \bk^{t}_i}{(i-1)!}d_i(x - y_t)  + \frac{\alpha_t}{A_t}\phi_{x_{t+1}, 1}(x)  + err_t\\
        \geq  \frac{1}{A_{t-1}}(f(x_{t+1}) + \langle \nabla f(x_{t+1}), x_t - x_{t+1}\rangle) +\sum \limits_{i=2}^{p+1} \left(\frac{1}{2}\right)^{i-2}\frac{ \bk^{t}_i}{(i-1)!}d_i(x - y_t)  \\
         + \frac{\alpha_t}{A_t}\Phi_{x_{t+1}, 1}(x)
        + \frac{\alpha_t}{A_t} \langle G_{x_{t+1}, 1} - \nabla f(x_{t+1}), x - x_{t+1} \rangle + err_t  
    \end{gather*}
    Next, we consider the sum of two linear models from the last inequality:
    \begin{gather*}
    \frac{1}{A_{t-1}}(f(x_{t+1}) + \langle \nabla f(x_{t+1}), x_t - x_{t+1}\rangle) + \frac{\alpha_t}{A_t}\Phi_{x_{t+1}, 1}(x) \\
    =\frac{1}{A_{t-1}}(f(x_{t+1}) + \langle \nabla f(x_{t+1}), x_t - x_{t+1}\rangle)\\
    + \frac{\alpha_t}{A_t} (f(x_{t+1}) 
    + \langle \nabla f(x_{t+1}), x - x_{t+1}\rangle)\\
        \stackrel{\eqref{eq:alphas}}{=} \frac{1 - \alpha_t}{A_t}f(x_{t+1}) + \frac{1 - \alpha_t}{A_t}\langle \nabla f(x_{t+1}), x_t - x_{t+1}\rangle + \frac{\alpha_t}{A_t}f(x_{t+1}) +\frac{\alpha_t}{A_t}\langle \nabla f(x_{t+1}), x- x_{t+1}\rangle\\
        \\
        \stackrel{\eqref{eq:u_t}}{=} \frac{f(x_{t+1})}{A_{t}} + \frac{1 - \alpha_t}{A_t}\langle \nabla f(x_{t+1}), \frac{v_t - \alpha_t y_t}{1 - \alpha_t}- x_{t+1}\rangle + \frac{\alpha_t}{A_t}\langle \nabla f(x_{t+1}), x- x_{t+1} \rangle \\
        = \frac{f(x_{t+1})}{A_{t}} + \frac{1}{A_t}\langle \nabla f(x_{t+1}), v_t- x_{t+1}\rangle + \frac{\alpha_t}{A_t}\langle \nabla f(x_{t+1}), x- y_t \rangle.
    \end{gather*}
    Therefore, 
    \begin{equation*}
        \begin{gathered}
            \psi_{t+1}(x) \geq \frac{f(x_{t+1})}{A_{t}} + \frac{1}{A_t}\langle \nabla f(x_{t+1}), v_t- x_{t+1}\rangle + \sum \limits_{i=2}^{p+1} \left(\frac{1}{2}\right)^{i-2}\frac{ \bk^{t}_i}{(i-1)!}d_i(x - y_t)\\
            + \frac{\alpha_t}{A_t}\langle \nabla f(x_{t+1}), x- y_t  \rangle + \frac{\alpha_t}{A_t}
           \langle G_{x_{i+1}, 1} - \nabla f(x_{i+1}), x - x_{i+1} \rangle + err_t.
        \end{gathered}
    \end{equation*}
    Finally, by \eqref{eq:estimating_seq}, we get
    \begin{gather*}
            \psi^{\ast}_{t+1} = \psi_{t+1}(y_{t+1})  \geq \frac{f(x_{t+1})}{A_{t}} + \frac{1}{A_t}\langle \nabla f(x_{t+1}), v_t- x_{t+1}\rangle 
            + \sum \limits_{i=2}^{p+1} \left(\frac{1}{2}\right)^{i-2}\frac{ \bk^{t}_i}{(i-1)!}d_i(y_{t+1} - y_t) \notag\\
            + \frac{\alpha_t}{A_t}\langle \nabla f(x_{t+1}), y_{t+1}- y_t  \rangle +\dk{err_{t+1}}.
    \end{gather*}
\end{proof}

Finally, we are in a position to prove the convergence rate theorem for the Accelerated Inexact Tensor Method (Algorithm \ref{alg:inexact_acc}). The proof uses the following technical assumption.
\begin{assumption}\label{eq:bound}
    Let $\{x_t, y_t\}_{t\geq 1}$ be generated from Algorithm \ref{alg:inexact_acc}. Then there exists $\bar{R} > 0$ such that
    \begin{equation}\label{eq:assumption_bound_x}
        \begin{aligned}
        &\aa{\|x_1 - x_0\| \leq \bar{R}, \quad \|x_t - x^{\ast}\| \leq \bar{R}} \quad \forall t \geq 1,\\  
        &\|y_t - x^{\ast}\| \leq \bar{R} \quad \forall t \geq 1.
        \end{aligned}
    \end{equation}
\end{assumption}
Note, that in the case of exact gradient, i.e. $\delta_1=0$, we do not need Assumption \ref{eq:bound}. Let also $R$ be such that
\begin{equation}
\label{R_teorem3}
\|x_0 - x^{\ast}\| \leq R.
\end{equation}

\begin{theorem}\label{thm:acc_convergence}
         Let Assumption \ref{eq:bound} hold, Condition \ref{cnd:sampling} be satisfied, $f(x)$ be convex and $p$ times differentiable function with Lipschitz constant $L_p$ for the $p$-th derivative, \aa{$\sigma \geq \max\lb 1; \tfrac{3}{p}\rb L_p
      $}, and $\eta_i\geq 3$, $i=1,\ldots,p$. 
    Assume also that   Algorithm  \ref{alg:inexact_acc} makes $T\geq 1$ iterations with parameters
    \begin{gather*}
        \bk_i^{t} = \frac{p}{2} \left( \frac{(i-1)6}{i\sqrt{5}}\right)^{i-1}(\eta_i + 2)\frac{\alpha_t^i}{A_t}\delta_i, i=2,...,p, \quad  \bk_{p+1}^{t} = \frac{1}{2}\left( \frac{6p }{\sqrt{5}(p+1)}\right)^{p} \frac{\alpha_t^{p+1}}{A_t}(2L_p+p\sigma).
    \end{gather*}
    
    If, for for all iterations $t \leq T$, we have 
    $$\delta_1 \leq \sum \limits_{i = 2}^p \frac{\delta_i}{(i-1)!}\|x_{t+1} - v_t\|^{i-1} + \frac{L_p}{p!}\|x_{t+1} - v_t\|^{p},$$
    then we have the following bound for the objective residual: 
    \begin{equation}
    \label{eq:acc_convergence}
    \begin{aligned}
       f(x_{T}) - f(x^{\ast}) 
       &\leq \aa{2}\delta_1 \aa{\bar{R}} +
       \sum \limits_{i = 2}^p \frac{\frac{p}{2} \left( \frac{6\aa{(i - 1)}}{\sqrt{5}\aa{i}}\right)^{i-1}(p+1)^{i}(\eta_i + 2)}{i!}\frac{\delta_i R^{i+1}}{(T+p+1)^{i}} \\ 
      &+ \frac{\left( \frac{6p }{\sqrt{5}}\right)^{p}}{p!} \frac{(\aa{2}L_p+p\sigma)R^{p+1}}{(T+p+1)^{p+1}},
     \end{aligned}
    \end{equation}
    or, in other words, 
    \begin{equation}
    \label{eq:acc_convergence_O}
    \begin{aligned}
       f(x_{T}) - f(x^{\ast}) 
       &\leq O(\delta_1 \aa{\bar{R}}) +
       \sum \limits_{i = 2}^p O\ls\frac{\delta_i R^{i+1}}{T^{i}}\rs + O\ls\frac{L_p R^{p+1}}{T^{p+1}}\rs.
     \end{aligned}
    \end{equation}
    Otherwise, if for some iteration $t \leq T$,
     $$\delta_1 \geq \sum \limits_{i = 2}^p \frac{\delta_i}{(i-1)!}\|x_{t+1} - v_t\|^{i-1} + \frac{L_p}{p!}\|x_{t+1} - v_t\|^{p},$$
     then
     \begin{equation}\label{eq:thm_acc_second_case}
         f(x_t) - f(x^{\ast}) \leq \ls\max\lb\frac{p\sigma}{L_p};\max_{\aa{i \in \lb 2,\ldots, p\rb}}\lb\eta_i\rb\rb + 2\rs\delta_1 \aa{\bar{R}}= O(\delta_1 \aa{\bar{R}}).
     \end{equation}

\end{theorem}
\begin{proof}
    \aa{Firstly, assuming that for all iterations $t \leq T$ it holds that 
    $$\delta_1 \leq \sum \limits_{i = 2}^p \frac{\delta_i}{(i-1)!}\|x_{t+1} - v_t\|^{i-1} + \frac{L_p}{p!}\|x_{t+1} - v_t\|^{p},$$}    
    we show by induction that, for all $t \geq 0$, 
    $$\frac{f(x_t)}{A_{t-1}}+ err_t \leq \psi^{\ast}_t,$$ where $err_t$ is defined in \eqref{eq:err_seq}.
    From \eqref{eq:psi_1}, since $A_0=1$, we have that $\frac{f(x_1)}{A_0} \leq \psi_1^{\ast}$. Let us assume that $$\frac{f(x_t)}{A_{t-1}} + err_t\leq \psi^{\ast}_t$$
    and show that  $\frac{f(x_{t+1})}{A_{t}} + err_{t+1}\leq \psi^{\ast}_{t+1}$. By Lemma \ref{lem:step}, we have 
    \begin{equation}\label{eq:show_pos}
        \begin{gathered}
            \psi_{t+1}^{\ast} \geq \frac{f(x_{t+1})}{A_t} + \frac{1}{A_t}\langle \nabla f(x_{t+1}), v_t - x_{t+1}\rangle 
            + \sum \limits_{i = 2}^{p+1} \left(\frac{1}{2} \right)^{i - 2} \frac{ \bk_i^{t}}{(i - 1)!}d_i(y_{t+1} - y_t) \\  + \frac{\alpha_t}{A_t}\langle \nabla f(x_{t+1}), y_{t+1} - y_t\rangle + \sum \limits_{j=1}^{t} \frac{\alpha_j}{A_j} \langle G_{x_{j+1}, 1} - \nabla f(x_{j+1}), y_{j+1} - x_{j+1} \rangle.   
        \end{gathered}
    \end{equation}
    Therefore, to complete the induction step we need to show, that the sum of all terms in the RHS except $\frac{f(x_{t+1})}{A_t}$ and error terms is non-negative. 
    
    Lemma \ref{lem:scalar_lb_cases} provides the lower bound for $\langle \nabla f(x_{t+1}), v_t - x_{t+1}\rangle$. Let us consider the case when the minimum in the RHS of \eqref{eq:scalar_lb_cases} is attained at the first term with particular $i = 2, \ldots, p$. 
    By Lemma \ref{lem:dual} with the following choice of the parameters
    $$z = y_t - y_{t+1}, ~~ v = \frac{\alpha_t}{A_t}\nabla f(x_{t+1}), ~~ \theta = \left(\frac{1}{2}\right)^{i-2}\frac{\bk_i^{t}}{(i-1)!},$$
    we have 
    \begin{equation}\label{eq:1case}
        \frac{1}{i}\left( \frac{1}{2} \right)^{i - 2}\frac{\bk_i^{t}}{(i-1)!}\|y_t -y_{t+1}\|^i + \frac{\alpha_t}{A_t}\langle \nabla f(x_{t+1}), y_{t+1} - y_t \rangle \geq - \frac{i - 1}{i}\left( \frac{\|\frac{\alpha_t}{A_t}\nabla f(x_{t+1})\|^i}{\left(\frac{1}{2}\right)^{i-2} \frac{\bk_i^{{t}}}{(i-1)!}} \right)^\frac{1}{i - 1}.
    \end{equation}
    Hence,
    \begin{gather*}
        \frac{f(x_{t+1})}{A_t} + \frac{1}{A_t}\langle \nabla f(x_{t+1}), v_t - x_{t+1}\rangle + \left(\frac{1}{{2}} \right)^{i - 2} \frac{ \bk_i^{{t}}}{(i - 1)!}d_i(y_{t+1} - y_t) + \frac{\alpha_t}{A_t}\langle \nabla f(x_{t+1}), y_{t+1} - {y_t}\rangle \\
        \stackrel{\eqref{eq:1case}}{\geq} \frac{f(x_{t+1})}{A_t} + \frac{1}{A_t}\langle \nabla f(x_{t+1}), v_t - x_{t+1}\rangle - \frac{i - 1}{i}\left( \frac{\|\frac{\alpha_t}{A_t}\nabla f(x_{t+1})\|^i}{\left(\frac{1}{{2}}\right)^{i-{2}} \frac{\bk_i^{{t}}}{{(i-1)}!}} \right)^\frac{1}{i - 1} \\
        \stackrel{\eqref{eq:scalar_lb_cases}}{\geq}
        \frac{f(x_{t+1})}{A_t} + \frac{1}{A_t}\frac{\sqrt{5}}{3}\|\nabla f(x_{t+1})\|^\frac{i}{i-1}\left( \frac{(i-1)!}{p(\eta_i + 2)\delta_i }\right)^\frac{1}{i-1} - \frac{i - 1}{i}\left( \frac{\|\frac{\alpha_t}{A_t}\nabla f(x_{t+1})\|^i}{\left(\frac{1}{{2}}\right)^{i-{2}} \frac{\bk_i^{{t}}}{{(i-1)}!}} \right)^\frac{1}{i - 1} \\
         \geq \frac{f(x_{t+1})}{A_t}, 
    \end{gather*}
    where the last inequality holds by our choice of the parameters
    \begin{equation}\label{eq:bar_kappa}
        \bk_i^{{t}} \geq \frac{p}{2} \left( \frac{(i-1)6}{i\sqrt{5}}\right)^{i-1}(\eta_i + 2)\frac{\alpha_t^i}{A_t}\delta_i,\quad \aa{i = 2, \ldots, p}.
    \end{equation}

    
    
    Next, we consider the case when the minimum in the RHS of \eqref{eq:scalar_lb_cases} is achieved on the second term. Again, by Lemma \ref{lem:dual} with the same choice of $z, v$ and with $\theta = {\left(\frac{1}{2}\right)^{p-1}\frac{\bk_{p+1}^{t}}{(p-1)!}}$, we have
    \begin{equation}
        \label{eq:2case}
        {\frac{1}{p}}\left( \frac{1}{{2}} \right)^{p - 1}\frac{\bk_{p+1}^{t}}{{\aa{p}!}}\|y_t -y_{t+1}\|^\aa{p+1} + \frac{\alpha_t}{A_t}\langle \nabla f(x_{t+1}), y_{t+1} - {y_t} \rangle \geq - \frac{p}{p+1}\left( \frac{\|\frac{\alpha_t}{A_t}\nabla f(x_{t+1})\|^{p+1}}{{\left(\frac{1}{2}\right)^{p-1} \frac{\bk_{p+1}^{t}}{(p-1)!}}} \right)^\frac{1}{p}.
    \end{equation}

   Hence, we get
    \begin{gather*}
        \frac{f(x_{t+1})}{A_t} + \frac{1}{A_t}\langle \nabla f(x_{t+1}), v_t - x_{t+1}\rangle + \frac{1}{2^{p-1}}\frac{\bk_{p+1}^{t}}{\aa{p}!}d_{p+1}(y_{t+1} - y_t) + \frac{\alpha_t}{A_t}\langle \nabla f(x_{t+1}), y_{t+1} - {y_t}\rangle\\
        \stackrel{\eqref{eq:2case}}{\geq} \frac{f(x_{t+1})}{A_t} + \frac{1}{A_t}\langle \nabla f(x_{t+1}), v_t - x_{t+1}\rangle  - \frac{p}{p+1}\left( \frac{\|\frac{\alpha_t}{A_t}\nabla f(x_{t+1})\|^{p+1}}{{\left(\frac{1}{2}\right)^{p-1} \frac{\bk_{p+1}^{t}}{\aa{p}!}}} \right)^\frac{1}{p}\\
        \stackrel{\eqref{eq:scalar_lb_cases}}{\geq}
        \frac{f(x_{t+1})}{A_t} + \frac{1}{A_t}\frac{\sqrt{5}}{3}\|\nabla f(x_{t+1})\|^\frac{p+1}{p}
        \left( \frac{(p-1)!}{(2L_p + p\sigma)}\right)^\frac{1}{p} -  \frac{p}{p+1}\left( \frac{\|\frac{\alpha_t}{A_t}\nabla f(x_{t+1})\|^{p+1}}{{\left(\frac{1}{2}\right)^{p-1} \frac{\bk_{p+1}^{t}}{\aa{p}!}}} \right)^\frac{1}{p}\\
         \geq \frac{f(x_{t+1})}{A_t},
    \end{gather*}
    where the last inequality holds by our choice of $\bk_{p+1}^{t}$:
    \begin{equation}\label{eq:bar_kappa_p+1}
        \bk_{p+1}^{t} \geq \frac{\aa{p}}{2}\left( \frac{6p }{\sqrt{5}(p+1)}\right)^{p} \frac{\alpha_t^{p+1}}{A_t}(2L_p+p\sigma).
    \end{equation}
    
    To sum up, by our choice of the parameters $\bk_{i}^{t}$, $i=2,...,p$, we obtain from  \eqref{eq:show_pos} that
    $$
        \psi_{t+1}^{\ast} \geq \frac{f(x_{t+1})}{A_t} + err_{t+1}.
    $$
    
    Thus, by induction, we obtain that, for all $t \geq 1$,
    \begin{gather}
       \frac{f(x_{t+1})}{A_t} + \sum \limits_{j=1}^{t}\frac{\alpha_j}{A_j}\langle G_{x_{j+1}, 1} - \nabla f(x_{j+1}), y_{j+1} - x_{j+1}\rangle \leq \psi_{t+1}^{\ast} \leq \psi_{t+1}(x^{\ast}) \stackrel{\text{Lemma } \ref{lem:upper_seq}}{\leq} \notag\\
       \frac{f(x^{\ast})}{A_{t}} + \delta_1\|x_1-x_0\| + \aa{\delta_1\|x^*-x_0\|}+ \sum \limits_{i = 2}^p \frac{ 2\delta_i + \bk^{t}_i}{(i - 1)!}d_i(x^{\ast} - x_0) + \frac{L_p + p\sigma +\bk_{p+1}^{t}}{p!}d_{p+1}(x^{\ast} - x_0) \notag\\
       + \sum \limits_{j = 1}^{t} \frac{\alpha_j}{A_j} \langle G_{x_{j+1}, 1} - \nabla f(x_{j+1}), x^{\ast} - x_{j+1}\rangle,
    \end{gather}
    where the first inequality has been just proved by induction and the third one comes from Lemma  \ref{lem:upper_seq}.
    
    We now bound the gradient error terms in this inequality:
    \begin{gather*}
        \sum \limits_{j = 1}^{t} \frac{\alpha_j}{A_j} \langle G_{x_{j+1}, 1} - \nabla f(x_{j+1}), x^{\ast} - x_{j+1}\rangle - \sum \limits_{j=1}^{t}\frac{\alpha_j}{A_j}\langle G_{x_{j+1}, 1} - \nabla f(x_{j+1}), y_{j+1} - x_{j+1}\rangle \\
        = \sum \limits_{j = 1}^{t } \frac{\alpha_j}{A_j} \langle G_{x_{j+1}, 1} - \nabla f(x_{j+1}), x^{\ast} - y_{j+1}\rangle 
        \leq \delta_1 \sum \limits_{j = 1}^{t} \frac{\alpha_j}{A_j} \|x^{\ast} - y_{j+1}\| \stackrel{\eqref{eq:assumption_bound_x}}{\leq}  \delta_1 \bar{R}\sum \limits_{j=1}^{t}\frac{\alpha_j}{A_j}.
    \end{gather*}

    Finally, we obtain the following convergence rate bound
    \begin{gather*}
       f(x_{T}) - f(x^{\ast})  \stackrel{\eqref{R_teorem3}, \eqref{eq:assumption_bound_x}}{\leq} A_T\left( \aa{2\delta_1\bar{R}}+\sum \limits_{i = 2}^p \frac{ 2\delta_i + \bk^{T}_i}{i!}R^{i} + \frac{ L_p + p\sigma + \bk_{p+1}^{T}}{(p+1)!}R^{p+1}
       + \delta_1  \bar{R}  \sum \limits_{j=1}^{T}\frac{\alpha_j}{A_j}\right) \\
      =  A_T\left(2\delta_1\bar{R} +  
      \sum \limits_{i = 2}^p \frac{ 2\delta_i + \frac{p}{2} \left( \frac{(i-1)6}{i\sqrt{5}}\right)^{i-1}(\eta_i + 2)\frac{\alpha_T^i}{A_T}\delta_i}{i!}R^{i}\right.\\ 
      \left.+ \frac{ L_p + p\sigma + \frac{1}{2}\left( \frac{6p }{\sqrt{5}(p+1)}\right)^{p} \frac{\alpha_T^{p+1}}{A_T}(2L_p+p\sigma)}{(p+1)!}R^{p+1}\right)
       + \delta_1  \bar{R}  A_T \sum \limits_{j=1}^{T}\frac{\alpha_j}{A_j} \\
    \stackrel{\eqref{alpha_t}, \eqref{eq:A_t_bound},\eqref{eq:ATat/At}}{\leq}
     \sum \limits_{i = 2}^p \frac{\aa{p} \left( \frac{6\aa{(i-1)}}{\sqrt{5}\aa{i}}\right)^{i-1}(p+1)^{i}(\eta_i + 2)}{i!}\frac{\delta_i R^{i}}{(T+p+1)^{i}} 
    + \frac{\left( \frac{6p }{\sqrt{5}}\right)^{p}}{\aa{(p-1)}!} \frac{(\aa{2}L_p+p\sigma)R^{p+1}}{(T+p+1)^{p+1}}
      + \aa{2}\delta_1 \bar{R}.
    \end{gather*}
    
     Now, for the rest of the proof we assume that for some iteration $t \leq T$ we have:
     $$\delta_1 \geq \sum \limits_{i = 2}^p \frac{\delta_i}{(i-1)!}\|x_{t+1} - v_t\|^{i-1} + \frac{L_p}{p!}\|x_{t+1} - v_t\|^{p}.$$
    Then, by Lemma \ref{lem:scalar_lb_cases} we have
    $$f(x_{t+1}) - f(x^{\ast}) \leq \ls\max\lb\frac{p\sigma}{L_p};\max_{i}\lb\eta_i\rb\rb + 2\rs\delta_1\|x_{t+1} - x^{\ast}\|.$$
    Thus, we obtain the second statement of the theorem.
\end{proof}

In Theorem \ref{thm:acc_convergence} has two cases. The first one provides an upper bound on the objective residual after $T$ iterations \eqref{eq:acc_convergence}. The second one shows, that there can be the case, when on iteration $t$ we have the objective residual smaller than $\delta_1\bar{R}$ already \eqref{eq:thm_acc_second_case}. Note, that this condition may be veryfied on every step of Algorithm \ref{alg:inexact_acc}.

The last term in the RHS of \ref{eq:acc_convergence} and \ref{eq:acc_convergence_O} corresponds to the case of an  Accelerated Tensor Method, i.e. $\delta_i = 0$, $i=1,\ldots,p$ provides similar convergence rate as \cite{nesterov2019implementable}. As in Inexact Tensor Method other terms correspond to the inexactness in derivatives. For $i=2,\ldots,p$ we have additional terms proportional to $\dfrac{\delta_i}{T^i}$, which are better than ones in non-accelerated case $\ls \dfrac{\delta_i}{T^{i-1}} \rs$. Note, that if the gradient error $\delta_1$ is non-zero, our bound does not guarantee the decrease of the objective below $O(\delta_1 \bar{R})$. In the case of exact gradient (i.e. $\delta_1 = 0$) we do not need additional Assumption $\ref{eq:bound}$. 

Assuming that we can control the errors in derivatives, we show how small one should take the errors to make the objective residual smaller than $\e$.

\begin{corollary}\label{lm:inexactenss_acc}
Let assumptions of Theorem \ref{thm:acc_convergence} hold and let $\e > 0$ be the desired solution accuracy.
Further, let the levels of inexactness in Condition \ref{cnd:sampling} satisfy the following inequalities:
\begin{equation}
    \delta_1 \leq \min \left\{\frac{\e}{2(p+1)\bar{R}}; \frac{\e}{(\max\lb\frac{p\sigma}{L_p};\max_{i}\lb\eta_i\rb\rb + 2)\aa{\bar{R}}} \right\} \aa{= O\left( \frac{\e}{\bar{R}}\right)},
\end{equation}
\begin{equation}
    \delta_i \leq \aa{\e^\frac{p+1-i}{p+1}\left(\frac{i!}{((p-1)!)^\frac{i}{p+1}}\frac{\left(\frac{6}{\sqrt{5}}\right)^\frac{p+1-i}{p+1}p^\frac{ip-p-1}{p+1}(2L_p + p\sigma)^\frac{i}{p+1} }{(p+1)^\frac{ip+p+1}{p+1}\left(\frac{i-1}{i} \right)^{i-1}(\eta_i+2)}\right)} = \aa{O\left(\e^\frac{p+1-i}{p+1} \right)}, \quad  i = 2, \ldots, p.
\end{equation}
Also, let the number of iterations $T$ of Algorithm \ref{alg:inexact} satisfy 
\begin{equation}
    T = \aa{\max\lb 1; \frac{R}{\e^\frac{1}{p+1}} \left(\frac{6p}{\sqrt{5}} \right)^\frac{p}{p+1}\left(\frac{p+1}{(p-1)!}\right)^\frac{1}{p+1}(2L_p + p\sigma)^\frac{1}{p+1} - p - 1\rb.}
\end{equation}
Then $x_{t}$ is an $\e$-solution of the problem \eqref{eq:PrSt}, i.e. $f(x_{t})-f(x^{\ast}) \leq \e$ for some $t = 1,\ldots, T$. 
\end{corollary}
\begin{proof}
    Firstly, let us consider the case in Theorem \ref{thm:acc_convergence} when 
    $$f(x_t) - f(x^{\ast}) \leq \ls\max\lb\frac{p\sigma}{L_p};\max_{i}\lb\eta_i\rb\rb + 2\rs\delta_1 \bar{R}.$$
    By our choice of $\delta_1$ we obtain
    $$f(x_t) - f(x^{\ast}) \leq \e.$$
    Secondly, we consider the case, when convergence rate is given by \eqref{eq:acc_convergence}.
    Then, by our choice of $\delta_i$ and $T$ 
    \begin{gather*}
         \aa{2}\delta_1 \bar{R} \leq \frac{\e}{p+1}, \\ 
         \frac{\aa{p} \left( \frac{6\aa{(i-1)}}{\sqrt{5}\aa{i}}\right)^{i-1}(p+1)^{i}(\eta_i + 2)}{i!}\frac{\delta_i R^{i+1}}{(T+p+1)^{i}} \leq  \frac{\e}{p+1},\\
         \frac{\left( \frac{6p }{\sqrt{5}}\right)^{p}}{p!} \frac{(\aa{2}L_p+p\sigma)R^{p+1}}{(T+p+1)^{p+1}}= \frac{\e}{p+1},
    \end{gather*}
    and, from \eqref{eq:acc_convergence}, we have
    \begin{equation*}
        f(x_{T+1}) - f(x^{\ast}) \leq \e.
    \end{equation*}
\end{proof}

Thus, the number of steps sufficient for the Accelerated Inexact Tensor Method to fin an $\e$-solution is the same as for accelerated exact tensor method \cite{nesterov2019implementable} and is equal to $O\ls \ls\dfrac{L_pD^{p+1}}{\e} \rs^\frac{1}{p+1}\rs$ under the assumption that $\delta_1 = O(\e)$ and $\delta_i = O\ls \e^\frac{p-i+1}{p+1}\rs$ for $i = 2, \ldots, p$. Note, that inexactness levers for $i = 2, \ldots, p$ are less strict than in non-accelerated case $O\ls \e^\frac{p-i+1}{p}\rs$.

\section{Implementation Details}\label{sec:implement}
\aa{Inexact tensor optimization methods, introduced in Sections \ref{sec:method}, \ref{sec:acceleration}, 
are based on the solution of auxiliary subproblems \eqref{eq:tensor_prox_algo}, \eqref{eq:acc_prox_x_step} in each iteration.} As we already proved, these problems are convex, and, therefore, can be solved by convex optimization methods. However, the complexity of solving these problems can slow down the convergence of inexact tensor methods. In this section, we show how to treat these problems in the particular case of $p=3$. To do that, we consider the third degree model which corresponds to $p=3$
\begin{equation}\label{eq:model_smooth}
    \zeta_x(h)\stackrel{\text{def}}{=} \omega_{x, 3} (x + h) =\phi_{x, 3}(x + h) + \frac{\delta_1}{2\gamma} + \frac{\gamma\delta_1}{2}\|h\|^2 + \frac{\delta_2}{2}\|h\|^2 + \frac{\delta_3}{6}\|h\|^3 + \frac{\sigma}{8}\|h\|^4.  
\end{equation}
We use the following notation for the first three approximate (sampled) derivatives 
$g = {G}_{x, 1}, ~ B = {G}_{x, 2}, ~ T = {G}_{x, 3} $, which leads to the following expression for the inexact Taylor polynomial: 
\begin{equation}\label{eq:approx_taylor_3ord}
    \phi_x(h)\stackrel{\text{def}}{=}  f\left({x}\right)+{g}^Th + \frac12 h^T{B}h + \frac{1}{6}{T}[h]^3.
\end{equation}

   

\begin{lemma}
For any $h \in \mathbb{R}^n$ and $\tau > 0$, we have
    \begin{equation}\label{eq:hessian_bound_smooth}
        - \frac1\tau B - \frac{1}{\tau}(\gamma\delta_1 + \delta_2)I - \delta_3\|h\|I - \frac{\tau}{2}L_3\|h\|^2I \preccurlyeq T[h] \preccurlyeq \frac1\tau B + \frac{1}{\tau}(\gamma\delta_1 + \delta_2)I + \delta_3\|h\|I + \frac{\tau}{2}L_3\|h\|^2I
    \end{equation}
\end{lemma}
\begin{proof}
    From \eqref{eq:hessian_bound} and definition of $\zeta_x(h)$ we obtain:
    \begin{gather*}
         0 \leq \langle \nabla^2f(x + h)u, u \rangle \leq \langle \nabla^2 \phi_x(h)u, u \rangle + \frac{L_3}{2}\|h\|^2\|u\|^2 + \delta_2\|u\|^2
         + \delta_3\|h\|\|u\|^2 \\
         \leq \langle (B + T[h]) u, u \rangle + \frac{L_3}{2}\|h\|^2\|u\|^2 + \delta_2\|u\|^2
         + \delta_3\|h\|\|u\|^2 
    \end{gather*}
    Replacing $h$ with $\tau h$ and dividing by $\tau$, we get
    \begin{gather*}
         - \langle T[h]u, u \rangle \leq \frac1\tau\langle Bu, u \rangle + \frac\tau 2 L_3\|h\|^2\|u\|^2 + 
         \frac 1\tau\delta_2 \|u\|^2 +  \delta_3 \|h\|\|u\|^2.
    \end{gather*}
    Replacing $h$ by $-h$ we obtain:
    \begin{gather*}
        \langle T[h]u, u \rangle \leq \frac1\tau\langle Bu, u \rangle + \frac\tau 2 L_3\|h\|^2\|u\|^2 +  \frac 1\tau\delta_2 \|u\|^2 +  \delta_3 \|h\|\|u\|^2. 
    \end{gather*}
    From the last two inequalities we get $\eqref{eq:hessian_bound_smooth}$.
   
\end{proof}

In the considered case of $p=3$,  Algorithm \ref{alg:inexact} requires to solve the following minimization problem at each iteration:
\begin{equation}\label{eq:aux_min_zeta}
      \min \limits_{h \in \mathbb{R}^n} \left\{ \zeta_x (h) = \phi_x(h) + \frac{\delta_1}{2\gamma} + (\gamma\delta_1 + \delta_2)d_2(h) + \delta_3d_3(h) + \frac{\sigma}{2}d_4(h)\right\}.
\end{equation}
For any $h \in \mathbb R^n:$
\begin{gather*}
    \nabla^2 \zeta_x (h) = \nabla^2 \phi_x(h) +
    (\gamma\delta_1 + \delta_2)\nabla^2 d_2(h) + \delta_3\nabla^2d_3(h) + \frac{\sigma}{2}\nabla^2d_{4}(h)\\
    = B + T[h]  + (\gamma\delta_1 + \delta_2)\nabla^2 d_2(h) + \delta_3\nabla^2d_3(h) + \frac{\sigma}{2}\nabla^2d_{4}(h)\\
    \stackrel{\eqref{eq:hessian_bound_smooth}}{\succcurlyeq}  
    B   + (\gamma\delta_1 + \delta_2)\nabla^2 d_2(h) + \delta_3\nabla^2d_3(h) + \frac{\sigma}{2}\nabla^2d_{4}(h) - \frac1\tau B - \frac{1}{\tau}(\gamma\delta_1 + \delta_2)I - \delta_3\|h\|I - \frac{\tau}{2}L_3\|h\|^2I \\
    \stackrel{\eqref{eq:ppf2deriv}}{\succcurlyeq} \left(1-\frac1\tau\right)B  + \left(1-\frac1\tau\right)(\gamma\delta_1 + \delta_2)\nabla^2d_2(h) +  \left(1-\frac1\tau\right)\delta_3\nabla^2d_3(h) + \ls \frac{\sigma - \tau L_3}{2}\rs\nabla^2d_4(h).
    \numberthis\label{eq:hessian_left_bnd}
\end{gather*}

On the other hand, for any $h \in \mathbb R^n$: 
\begin{gather*}
    \nabla^2 \zeta_x (h) =\nabla^2 \phi_x(h) +
    (\gamma\delta_1 + \delta_2)\nabla^2 d_2(h) + \delta_3\nabla^2d_3(h) + \frac{\sigma}{2}\nabla^2d_{4}(h)\\
    \stackrel{\eqref{eq:hessian_bound_smooth}}{\preccurlyeq} B + \frac1\tau B + \frac{1}{\tau}(\gamma\delta_1 + \delta_2)I + \delta_3\|h\|I + \frac{\tau}{2}L_3\|h\|^2I +
    (\gamma\delta_1 + \delta_2)\nabla^2 d_2(h) + \delta_3\nabla^2d_3(h) + \frac{\sigma}{2}\nabla^2d_{4}(h)\\
    \stackrel{\eqref{eq:ppf2deriv}}{\preccurlyeq} \left(1+\frac1\tau\right)B  + \left(1+\frac1\tau\right)(\gamma\delta_1 + \delta_2)\nabla^2d_2(h) +  \left(1+\frac1\tau\right)\delta_3\nabla^2d_3(h) + \ls \frac{\sigma + \tau L_3}{2}\rs\nabla^2d_4(h).
    \numberthis\label{eq:hessian_right_bnd}
\end{gather*}

Next, we set $\sigma = \tau^2 L_3$ with $\tau \geq 1$. In view of Theorem \ref{thm:model_cnvxty}, model $\zeta_x(h)$ is convex with this choice of $\sigma$.

Then, from \eqref{eq:hessian_left_bnd}, \eqref{eq:hessian_right_bnd} we obtain
\begin{gather*}
    \left(1-\frac1\tau\right)B  + \left(1-\frac1\tau\right)(\gamma\delta_1 + \delta_2)\nabla^2d_2(h) +  \left(1-\frac1\tau\right)\delta_3\nabla^2d_3(h) + \ls \frac{\sigma - \tau L_3}{2}\rs\nabla^2d_4(h)\preccurlyeq \nabla^2 \zeta_x(h) \\
    \preccurlyeq \frac{1+\tau}{1-\tau}\left(\left(1-\frac1\tau\right)B  + \left(1-\frac1\tau\right)(\gamma\delta_1 + \delta_2)\nabla^2d_2(h) +  \left(1-\frac1\tau\right)\delta_3\nabla^2d_3(h) + \ls \frac{\sigma - \tau L_3}{2}\rs\nabla^2d_4(h) \right).
\end{gather*}

    Let $\rho_x(h) = \frac{1}{2} \left(1 - \frac{1}{\tau}\right)\langle B h, h \rangle+\left(1-\frac1\tau\right)(\gamma\delta_1 + \delta_2)d_2(h) +  \left(1-\frac1\tau\right)\delta_3 d_3(h) + \ls \frac{\sigma - \tau L_3}{2}\rs d_4(h)$. Therefore, we have proved the  relative strong convexity with constant $1$ and relative smoothness with constant $\kappa(\tau) =\frac{1 + \tau}{1 - \tau}$ of the function $\zeta_x(h)$ with respect to the function $\rho_x(h)$ \cite{lu2018relatively}:
    \begin{equation}
        \nabla^2 \rho_x(h) \preccurlyeq \nabla^2 \zeta_x(h) \preccurlyeq \frac{1+\tau}{1-\tau} \nabla^2 \rho_x(h)  .
    \end{equation}


The relative smoothness condition allows to  solve the auxiliary problem \eqref{eq:aux_min_zeta} very efficiently \cite{nesterov2019implementable,lu2018relatively} by the iterates
\begin{equation}\label{eq:sub_algo}
    h_{k+1}=\arg \min _{h \in \mathbb{R}^n}\left\{\left\langle\nabla \zeta_x(h_{k}), h \right\rangle+\kappa(\tau) \beta_{\rho_{x}}\left(h_{k}, h\right)\right\}
\end{equation}
with linear rate of convergence. 

\aa{Note, that in the accelerated method (Algorithm \ref{alg:inexact_acc}) subproblem \eqref{eq:acc_prox_x_step} is different than minimizing the model $\zeta_x(h)$ and has the form:
$$\min\limits_x\lb \phi_{x,3}(x+h) +  \frac{\delta_2}{2}\|h\|^2 + \frac{\delta_3}{6}\|h\|^3 + \frac{\sigma}{8}\|h\|^4 \rb.$$
However, this objective function also satisfies relative strong convexity and relative smoothness conditions. Indeed, the term $\frac{\delta_1}{2\gamma}$ does not correspond to convexity or smoothness. Therefore, we can omit it and the analysis will be the same. We set $\gamma = 0$ in equation \eqref{eq:model_smooth} and the reference function
$$\rho_x(h) = \frac{1}{2} \left(1 - \frac{1}{\tau}\right)\langle B h, h \rangle+\left(1-\frac1\tau\right)\delta_2d_2(h) +  \left(1-\frac1\tau\right)\delta_3 d_3(h) + \ls \frac{\sigma - \tau L_3}{2}\rs d_4(h).$$
Then, the objective in subproblem \eqref{eq:acc_prox_x_step} also satisfies relative smoothness and strong convexity condition. Therefore, it can also be  solved by the process \eqref{eq:sub_algo}.
}

 According to \cite{nesterov2019implementable} it is not necessary to calculate the full third derivative tensor $T$ in the above derivations. It is sufficient to use an automatic differentiation technique to calculate third-order derivative in a certain direction.

\section{Stochastic Tensor and Accelerated Tensor Methods} \label{sec:smpl}
 In this section, we apply Inexact Tensor Method and Accelerated Inexact Tensor Method to solve stochastic optimization problem in the online \eqref{eq:online_intro} and offline \eqref{eq:offline_intro} settings. The main step to do that is to find sufficient conditions for Condition \ref{cnd:sampling} to be satisfied in these two settings.  
To do that we first introduce an additional assumption on the objective function $f$:
    \begin{assumption}\label{as:add_lip}
        The function $f(x)$ and its derivatives $\nabla f(x), \ldots , \nabla^{p-1}f(x)$ are Lipschitz continuous:
        $$\|\nabla^i f(x) - \nabla^i f(y)\| \leq L_i \|x - y\|, i=0, \ldots, p-1.$$
    \end{assumption}

In the stochastic versions of Algorithms \ref{alg:inexact}, \ref{alg:inexact_acc} for the stochastic optimization problem \eqref{eq:PrSt} in the online \eqref{eq:online_intro} and offline \eqref{eq:offline_intro} settings we sample stochastic derivatives in each iteration in order to form mini-batch approximations for the derivatives of $f$. 

More precisely, for $\mathcal{S}_1, \mathcal{S}_2, \ldots, \mathcal{S}_p$ being sample sets for each derivative, we set
\begin{equation*}
\begin{aligned}
    {G}_{{x}, i} = \frac{1}{\left|\mathcal{S}_i\right|} \sum_{j \in \mathcal{S}_i} \nabla^i f\left({x}, \xi_j\right), \quad i=1,..,p.
    \end{aligned}
\end{equation*}

In the next subsections we show how to choose the size of the sample sets to satisfy Condition $\ref{cnd:sampling}$ for both online and offline settings.
    
\subsection{Online Setting}
For the online setting, i.e. when \eqref{eq:online_intro} holds, we need one more assumption to be able to satisfy Condition $\ref{cnd:sampling}$.
\begin{assumption}\label{as:stoch}
    For all $i = 1, 2, \ldots, p, ~ \xi,$ and $ x \in \R^n$:
    $$
    \| \nabla^i f({x}, \xi)-\nabla^i f({x}) \| \leq M_i.
    $$
\end{assumption}
     
      From the following tensor concentration bound theorem we derive the required conditions.

    \begin{theorem}[Tensor Hoeffding Inequality \cite{lucchi2019stochastic}]
         Let $\mathcal{X}$ be a sum of $m$ i.i.d. tensors $\mathcal{Y}_{i} \in \mathbb{R}^{d_{1} \times \cdots \times d_{k}}$. Let ${u}_{1}, \ldots {u}_{k}$ be such that $\left\|{u}_{i}\right\|=1$ and assume that for each tensor $\mathcal{Y}_{i}$, it holds that $ a \leq \mathcal{Y}_{i}\left({u}_{1}, \ldots, {u}_{k}\right) \leq$b. Let
         $\sigma:=(b-a)$.  Then we have 
            $$
            P(\|\mathcal{X}-\mathbb{E} \mathcal{X}\| \geq t) \leq k_{0}^{\left(\sum_{i=1}^{k} d_{i}\right)} \cdot 2 \exp \left(-\frac{t^{2}}{2 m \sigma^{2}}\right)
            $$
        where $k_{0}=\left(\frac{2 k}{\log (3 / 2)}\right)$.
    \end{theorem}

    This Theorem allows us to provide a sufficient condition on the sample sets  $\mathcal{S}_i$ in order to satisfy Condition $\ref{cnd:sampling}$. 
    \begin{lemma}\label{lem:batch_size_online} Let Assumptions \ref{as:lip}, \ref{as:add_lip}, \ref{as:stoch} be satisfied. Then, for any fixed small constants $\delta_i > 0$ we can choose the sizes of the sample sets $\mathcal{S}_i$ 
    in equation $\eqref{eq:sampled_derivs}$ to be 
            $$\left|\mathcal{S}_i\right| = n_i = \tilde{\mathcal{O}}\left((L_{i - 1} + M_i)^{2} \cdot \delta_i^{-2}\right), i=1, \ldots, p 
            $$
    so that with probability at least $1 - \delta$ Condition \ref{cnd:sampling} is satisfied. 
    \end{lemma}
    
    \begin{proof}
        Using Assumptions \ref{as:lip}, \ref{as:stoch} and the triangle inequality, we obtain 
        $$
        \|{G}_{x, i}\| = \frac{1}{n_i} \sum \limits_{j = 1}^{n_i} \|\nabla^i f(x, \xi_j) \| \leq \frac{1}{n_i} \sum \limits_{j = 1}^{n_i} \left( \| \nabla^i f(x, \xi_j) - \nabla^i f(x) \| + \|\nabla^i f(x) \|\right) \leq M_i + L_{i-1} \stackrel{\text{def}}{=} \sigma_i.
        $$
        Then, the proof completely repeats the proof of Lemma 11 in \cite{lucchi2019stochastic}. 
        We require the probability of a large deviation to be smaller than $\delta \in (0, 1]$:
        $$
            \mathbf{P}\left\{\|G_{x, i}-\nabla^{i} f({x})\right\|>t\} \leq k_{0}^{ni 
            } \cdot 2 \exp \left(-\frac{t^{2} n_{i}}{2 \sigma_{i}^{2}}\right) \leq \delta.
        $$
        Taking the logarithm of both sides, we get 
        $$
            -\frac{t^{2} n_{i}}{2 \sigma_{i}^{2}} \leq \log \frac{\delta}{2 k_{0}^{ni}}
        $$
        which is equivalent to
        $$
            n_{i} \geq \frac{2 \sigma_{i}^{2}}{t^{2}} \log \frac{2 k_{0}^{ni}}{\delta}.
        $$
        Finally, we can simply choose $t=\delta_i$ in order to satisfy  \eqref{eq:sampling_cnd} in Condition $\ref{cnd:sampling}$ since then, with probability at least $1-\delta$, $\forall y \in \mathbb{R}^n$,
        $$
            \begin{aligned}
                \left\|{G}_{x, i}[{y - x}]^{i-1}-\nabla^{i} f({x})[y - x]^{i-1}\right\|
                 \leq\left\|{G}_{x, i}-\nabla^{i} f({x})\right\|\|y - x\|^{i-1} 
                 \leq \delta_i\|y - x\|^{i-1}.
            \end{aligned}
        $$
    \end{proof}
    
    Next, we obtain the required batch sizes for our stochastic tensor methods in order to achieve convergence rates as in the case of exact derivatives. The following result immediately follows from the Lemma above and Corollary \ref{lm:inexactenss_nonacc}.  
    
    \begin{corollary}\label{col:online}
        Let Assumptions \ref{as:lip}, \ref{as:add_lip}, \ref{as:stoch} be satisfied. Then, we can choose the sizes $|\mathcal{S}_i|$ of sample sets $\mathcal{S}_i$ in \eqref{eq:sampled_derivs} to be 
        $$n_i = |\mathcal{S}_i| = \tilde{\mathcal{O}}\left((L_{i - 1} + M_i)^{2} \cdot {\e^{-2\frac{p -i + 1}{p}}}\right), i=1\ldots p$$
        so that with probability at least $1 - \delta$ Stochastic Tensor Method has complexity $T = O\left(\frac{1}{\e^{1/p}}\right)$.
    \end{corollary}
    
    A similar result can be obtained for the Accelerated Stochastic Tensor Method.
    
    \begin{corollary}\label{col:online_acc}
        Let Assumptions \ref{as:lip}, \ref{as:add_lip}, \ref{as:stoch} be satisfied. Then, we can choose the sizes $|\mathcal{S}_i|$ of sample sets $\mathcal{S}_i$ in \eqref{eq:sampled_derivs} to be 
        \begin{gather}
             n_1 = |\mathcal{S}_1| = \tilde{\mathcal{O}}\left({(L_0 + M_1)^{2}} \cdot {\e^{-2}}\right),\\
            n_i = |\mathcal{S}_i| = \tilde{\mathcal{O}}\left({(L_{i - 1} + M_i)^{2}} \cdot 
            {\e^{-2\frac{p -i + 1}{p + 1}}}\right), i=2, \ldots, p
        \end{gather}
        so that with probability at least $1 - \delta$ Accelerated Stochastic Tensor Method has complexity $T = O\left(\frac{1}{\e^{1/(p + 1)}}\right)$.
    \end{corollary}
    
    Corollaries \ref{col:online}, \ref{col:online_acc} show how to choose batch sizes for derivatives in inexact tensor methods to achieve convergence as 
    in exact case for online setting. Gradient sample sizes of non-accelerated and accelerated tensor methods are the same. For the case of higher derivatives (of order $\geq 2$), we need to use fewer samples for the accelerated method than for the non-accelerated one. Moreover, batch sizes are decreasing with the growth of derivatives order.

\subsection{Offline setting}
For the offline setting, i.e. when \eqref{eq:offline_intro} holds, we use the following result to provide a sufficient condition for Condition $\ref{cnd:sampling}$ to hold.
\begin{theorem}[Tensor Hoeffding--Serfling Inequality \cite{lucchi2019stochastic}]\label{thm:hoeff-serfl}
        Let $\mathcal X$ be a sum of $m$ tensors $\mathcal Y_i \in \mathbb R^{d_1 \times ... \times d_k}$, sampled without replacement from a finite population $\mathcal A$ of size $N$. Let ${u}_i, ..., {u}_k$ be such that $\|{u}_i\| = 1$ and assume that for each tensor $i$, $a \le \mathcal Y_i({u}_i, ..., {u}_k) \le b$. Let $\sigma := (b - a)$, then we have 
        \begin{equation}\label{eq:tensor_hoeffding-serfling}
            P ( \| \mathcal X - \mathbb E \mathcal X \| \ge t) \le 
            k_0^{\sum_{i = 1}^k d_i} \cdot 2 \exp \bigg( - \frac{t^2 m^2}{2 \sigma^2 (m + 1) (1 - m / N)} \bigg),
        \end{equation}
        where $k_0 = \frac{2k}{\log(3 / 2)}$.
    \end{theorem}
    
    The next lemma follows from Theorem \ref{thm:hoeff-serfl}. The proof can be found in \cite{lucchi2019stochastic}.
    
    \begin{lemma}\label{lem:batch_size_offline} Let Assumptions \ref{as:lip}, \ref{as:add_lip} be satisfied. Then, for any fixed small constants $\delta_i > 0$ we can choose the sizes $\left|\mathcal{S}_i\right|$ of sample sets $\mathcal{S}_i$ in \eqref{eq:sampled_derivs} to be 
            $$n_i = \left|\mathcal{S}_i\right|= \tilde{\mathcal{O}}\left( \frac{\delta_i^2
            }{L_{i - 1}^2} + \frac{1}{{m}} \right)^{-1}$$
    so that with probability at least $1 - \delta$ Condition \ref{cnd:sampling} holds.
    \end{lemma}
    
    Similar to the previous subsection we obtain the following results. 
    
    \begin{corollary}\label{col:offline}
        Let Assumptions \ref{as:lip}, \ref{as:add_lip} be satisfied. Then, we can choose the sizes $|\mathcal{S}_i|$ of sample sets $\mathcal{S}_i$ in \eqref{eq:sampled_derivs} to be 
        $$n_i = |\mathcal{S}_i| =
        \tilde{\mathcal{O}}\left( \frac{\e^{2\frac{p -i + 1}{p}}
            }{L_{i - 1}^2} + \frac{1}{{m}} \right)^{-1}, i=1\ldots n
        $$
        so that with probability at least $1 - \delta$ Stochastic Tensor Method has complexity $T = O\left(\frac{1}{\e^{1/p}}\right)$.
    \end{corollary}

    \begin{corollary}\label{col:offline_acc}
        Let Assumptions \ref{as:lip}, \ref{as:add_lip} be satisfied. Then, we can choose the sizes $|\mathcal{S}_i|$ of sample sets $\mathcal{S}_i$ in \eqref{eq:sampled_derivs} to be 
        $$n_i = |\mathcal{S}_i| = \tilde{\mathcal{O}}\left( \frac{\e^{2\frac{p -i + 1}{p + 1}}
            }{L_{i - 1}^2} + \frac{1}{{m}} \right)^{-1}, i=2\ldots n$$
        so that with probability at least $1 - \delta$ Accelerated Stochastic Tensor Method has complexity $T = O\left(\frac{1}{\e^{1/(p + 1)}}\right)$.
    \end{corollary}
    
    Corollaries \ref{col:offline}, \ref{col:offline_acc} are the counterpart of Corollaries \ref{col:online}, \ref{col:online_acc} and shows the analogical result the offline setting. 
   
\section*{Acknowledgments}
The work of A. Agafonov, and D. Kamzolov
was supported by a grant for research centers in the field of artificial intelligence, provided by the Analytical Center for the Government of the Russian Federation in accordance with the subsidy agreement (agreement identifier 000000D730321P5Q0002) and the agreement with the Moscow Institute of Physics and Technology dated November 1, 2021 No. 70-2021-00138.

The authors are very grateful to Yu.E. Nesterov, \fbox{A.I. Golikov}, Yu.G. Evtushenko and G. Scutari for fruitful discussions.
    
\bibliography{bibliography}

\end{document}